\documentclass[a4paper,10pt]{amsart}
\usepackage[utf8x]{inputenc}
\usepackage{graphicx, xcolor}
\usepackage{amsmath,amssymb,amsthm,comment,mathtools}

\usepackage[colorlinks,allcolors=blue]{hyperref} 
\usepackage{cleveref}
\usepackage{pdfpages}
\usepackage{amsmath}
\usepackage{subfigure,cite}
\usepackage{thm-restate}

\usepackage[a4paper,margin=2.6cm]{geometry}

\newtheorem{thm}{Theorem}[section]
\newtheorem{cor}[thm]{Corollary}

\newtheorem{prop}[thm]{Proposition}
\newtheorem{lem}[thm]{Lemma}
\newtheorem{rem}[thm]{Remark}

\newtheorem{conj}[thm]{Conjecture}

\newtheorem{ex}[thm]{Example}

\theoremstyle{definition}
\newtheorem{defi}[thm]{Definition}

\def\mult{{\rm mult}}
\def\deg{{\rm deg}}
\def\X{\mathbf{x}}
\def\Y{\mathbf{y}}
\def\T{\mathbf{t}}
\def\U{\mathcal{U}}
\def\Um{\mathcal U_m}
\def\Umn{\mathcal U_m^{\,n}}
\def\roots{\{1,\varepsilon,\varepsilon^2,\dots,\varepsilon^{m-1}\}}

\title[Sensitivity of $m$-ary functions]{Sensitivity of $m$-ary functions and low degree partitions of Hamming graphs}

 \author[S. Asensio]{Sara Asensio}
 \address{Instituto de Investigaci\'on en Matem\'aticas de la Universidad de Valladolid (IMUVa), Universidad de Valladolid, 47011, Valladolid, Spain.}
 \email{sara.asensio@uva.es}

\author[I. García-Marco]{Ignacio García-Marco}
 \address{Instituto de Matem\'aticas y Aplicaciones (IMAULL), Secci\'on de Matem\'aticas, Facultad de
Ciencias, Universidad de La Laguna, 38200, La Laguna, Spain}
 \email{iggarcia@ull.edu.es}

\author[K. Knauer]{Kolja Knauer}
\address{Aix Marseille Univ, Universit\'e de Toulon, CNRS, LIS, Marseille, France\\Departament de Matem\`atiques i Inform\`atica,
Universitat de Barcelona, Spain}
 \email{kolja.knauer@ub.edu}

\thanks{
This work was supported in part by the grant PID2022-137283NB-C22 funded by MICIU/AEI/10.13039/501100011033 and by ERDF/EU.
The first author is also supported by the UVa 2023 call for predoctoral contracts, cofunded by Banco Santander.
The third author is partially supported by the Spanish Ministerio de Econom\'ia, Industria y Competitividad through grant RYC-2017-22701 and the Severo Ochoa and Mar\'ia de Maeztu Program for Centers and Units of Excellence in R\&D (CEX2020-001084-M)}

\keywords{$m$-ary function, complexity measures, sensitivity, Hamming graphs}

\begin{document}

\begin{abstract}
The study of complexity measures of Boolean functions led Nisan and Szegedy to state the sensitivity conjecture in 1994, claiming a polynomial relation between degree and sensitivity. This problem remained unsolved until 2019, when Huang proved the conjecture via an equivalent graph theoretical reformulation due to Gotsman and Linial. 

We study \emph{$m$-ary functions}, i.e., functions $f: T^n \rightarrow T$ where $T\subseteq \mathbb{C}$ is a finite alphabet of cardinality $|T| = m $.
We extend the notions of \emph{degree} $\deg(f)$ and \emph{sensitivity} $s(f)$ to $m$-ary functions and show $s(f)\in O(\deg(f)^2)$. 
This generalizes results of Nisan and Szegedy.
Conversely, we introduce the $m$-ary sensitivity conjecture, claiming a polynomial upper bound for $\deg(f)$ in terms of $s(f)$. 
Analogously to results of Gotsman and Linial, we provide a formulation of the conjecture in terms of imbalanced partitions of Hamming graphs into low degree subgraphs. Combining this with ideas of Chung, F\"uredi, Graham and Seymour, we show that for any prime $p$ the bound in the $p$-ary sensitivity conjecture has to be at least quadratic: there exist $p$-ary functions $f$ of arbitrarily large degree and $\deg(f)\in \Omega(s(f)^2)$.
\end{abstract}

\maketitle

\section{Introduction}
Many computational combinatorial problems can be thought of as evaluating a \emph{Boolean function}, i.e., a function $f:\{0,1\}^n\rightarrow\{0,1\}$. Hence, determining the complexity of (evaluating) a Boolean function is one of the most fundamental algorithmic problems of Theoretical Computer Science. This framework was proposed in the 60s~\cite{MP69} and first measures of complexity were studied and compared in the 80s~\cite{KKL88} and 90s~\cite{NS94}. Nowadays, Boolean functions intervene in such diverse areas as Cryptography~\cite{Car10}, Quantum Computing~\cite{AMP24,KLM07} and Deep Learning~\cite{HCP21,NTZAFLV22}.

A variety of complexity measures of  Boolean functions has been considered in the literature, see~\cite{BW02} for a survey. Already since the 90s it has been known that most deterministic measures are polynomially equivalent except (back then) possibly the  sensitivity. This led Nisan and Szegedy~\cite{NS94} to the \emph{sensitivity conjecture}, which claims that the sensitivity is polynomially equivalent to the other measures. A first step is an equivalence theorem due to Gotsman and Linial~\cite{GL92}, which translates the sensitivity conjecture to a problem in graph theory. In 2019, Huang~\cite{Hua19} gave a one-page proof for this statement based on elementary linear algebra arguments. Namely, he showed that any induced subgraph on more than half of the vertices of the $n$-dimensional hypercube has maximum degree at least $\sqrt{n}$. See~\cite{Ase23} for an extended summary of this story (in Spanish). 

Studying Huang's result in other families of graphs is an active area of research. His result has been generalized to Cartesian powers of directed cycles (Tikaradze,~\cite{Tik22}), Cartesian powers of paths (Zeng and Hou,~\cite{ZH24}), and other Cartesian and semistrong products of graphs (Hong, Lai and Liu,~\cite{HLL20}). Alon and Zheng showed that Huang's result implies a similar result for Cayley graphs over $\mathbb{Z}_2^{\,n}$~\cite{AZ20}, which was later generalized to arbitrary abelian Cayley graphs by Potechin and Tsang~\cite{PT20}, and to Cayley graphs of Coxeter groups and expander graphs by García-Marco and Knauer~\cite{GMK22}. Similar results on Kneser graphs have been developed by Frankl and Kupavskii~\cite{FK20}, and Chau, Ellis, Friedgut and Lifshitz~\cite{CEFL23}. On the negative side, infinite families of Cayley graphs with low-degree induced subgraphs on many (more than the independence number) vertices were constructed by Lehner and Verret~\cite{LV20}, and García-Marco and Knauer~\cite{GMK22}. The existence of such subgraphs in Hamming graphs $H(n,m)$ is the topic of several recent works, see~\cite{Don21, Tan22, GMK22, PT24}. 
 
The study of sensitivity has also been extended beyond Boolean functions without going through graphs. Dafni, Filmus, Lifshitz, Lindzey and Vinyals~\cite{DFLLV21} consider $f:\mathcal{X}\rightarrow\{0,1\}$ on different domains such as the symmetric group $\mathcal{X}=S_n$. They show that in this case all classical complexity measures of Boolean functions can also be defined and are polynomially equivalent. In particular, they prove the analogous result to the sensitivity conjecture. 

Another natural generalization of Boolean functions are {\em $m$-ary functions}, i.e., functions $f:T^n\rightarrow T$ where $T\subseteq \mathbb{C}$ is a finite set of cardinality $m$. For $p$ prime, $p$-ary functions have been studied in relation to combinatorial structures such as association schemes~\cite{WHL22} and strongly regular graphs~\cite{CTZ11, HL19, MJ19}. Moreover, they provide connections to Cryptography and Coding Theory, see~\cite{Mei22}. Going back to the initial motivation of Boolean functions, $m$-ary functions can be viewed as combinatorial problems over a multi-valued logic.

\subsection*{Our results:}
The aim of the present paper is to study the sensitivity of $m$-ary functions.  
In \Cref{measures} we define complexity measures of an $m$-ary function $f$, namely the sensitivity $s(f)$, the block sensitivity $bs(f)$ and the degree $\deg(f)$. After this, in~\Cref{upperbound} we prove a polynomial upper bound for the sensitivity of an $m$-ary function in terms of its degree; more precisely, we prove 

\begin{restatable}{thm}{thcotafacil}
\label{th:cotafacil}
   For every $m$-ary function $f$, we have $s(f)\leq 2\ (m-1)^3\ \deg(f)^2.$
\end{restatable}

As in the Boolean case, the most difficult question seems to be the proof of a polynomial bound in the other direction, prompting us to put forward the 

\begin{restatable}[\bf $m$-ary Sensitivity Conjecture]{conj}{conjsensit}
\label{conj:sensit}
    For every $m$ there exists a constant $c$ such that $\deg(f)\in \mathcal O(s(f)^c)$ for every $m$-ary function $f$.
\end{restatable}

As a first step towards the conjecture, we generalize the equivalence theorem by Gotsman and Linial~\cite{GL92}. This leads to a graph theoretical formulation of~\Cref{conj:sensit} in terms of the Hamming graph $H(n,m)$, see~\Cref{thm:equivalence}. Even if~\Cref{thm:equivalence} is more technical than the Boolean case, we believe that it will be crucial in an eventual resolution of~\Cref{conj:sensit}. As of now,~\Cref{thm:equivalence} has two main consequences. First, we present a  far generalization of a result by Chung, Füredi, Graham and Seymour~\cite{CFGS88}. For every prime $p$, we present a construction proving that the polynomial upper bound in~\Cref{conj:sensit} has to be at least quadratic for $p$-ary functions. 
\begin{restatable}{thm}{corquadratic}
\label{cor:quadratic}
   For every set $T\subseteq \mathbb{C}$ of prime cardinality $|T| = p$ and every positive integer $D$, there exists a  $p$-ary function $f:T^n\rightarrow T$ with $\deg(f) > D$ and  $s(f)^2 / (p-1)^3\leq\deg(f)$.
\end{restatable}

Second, in~\Cref{last}, based on~\Cref{thm:equivalence} we obtain a natural graph theoretical strengthening of~\Cref{conj:sensit} in terms of imbalanced partitions of the Hamming graph $H(n,m)$ into induced subgraphs of low maximum degree, see~\Cref{conj:stronger}. Already the first open case is non-trivial and tempting, where $\Delta$ denotes the maximum degree of a graph.

\begin{restatable}{conj}{conjstrongertres}
\label{conj:strongertres} There exists $\mu > 0$ such that for any partition of the Hamming graph $H(n,3)$ into three (possibly empty) induced subgraphs $H_1,H_2,H_{3}$ not all of the same order, it holds: $$\max\{\Delta(H_{1}),\Delta(H_{2}),\Delta(H_3)\} \in \Omega(n^{\mu})\ .$$
\end{restatable}
\bigskip

\section{Some complexity measures of $m$-ary functions}\label{measures}

For the entire paper let $m,n$ be positive integers and $T \subseteq \mathbb C$ a set of size $m$. For convenience, we consider $m$-ary functions over $T\subseteq \mathbb C$, i.e., maps $f: T^n \rightarrow T$. In this section we introduce several complexity measures of $m$-ary functions, which naturally extend those of Boolean functions. We begin with the first ingredient of the sensitivity conjecture.

\begin{defi}[Degree]\label{def:degree} A polynomial $F \in \mathbb C[x_1,\ldots,x_n]$ {\em represents} an $m$-ary function $f: T^n \rightarrow T$ if $F(\mathbf{x}) = f(\mathbf{x})$ for all $\mathbf{x}\in T^n$.
The {\em degree} of $f$ is $\deg(f)=\min\{\deg(F)\mid F \in \mathbb C[x_1,\ldots,x_n]\text{ represents }f\}$, the smallest degree of a polynomial representing $f$.
\end{defi}

The following polynomial interpolation result shows a way to compute the degree of an $m$-ary function:

\begin{prop}\label{pr:uniquepoly} Let $T \subseteq \mathbb C$ be a set of size $m$ and $f: T^n \rightarrow T$ an $m$-ary function. There is a unique polynomial $F \in \mathbb C[x_1,\ldots,x_n]$ of degree at most $m-1$ in each variable representing~$f$. Moreover, $\deg(f) = \deg(F)$.
\end{prop}
\begin{proof}
Let $f: T^n \rightarrow T$ be an $m$-ary function. By, e.g,  \cite[Chapter 6.6]{Isa66} or \cite[Section 19]{Ste27}, there is a unique polynomial $F \in \mathbb C[x_1,\ldots,x_n]$ of degree at most $m-1$ in each variable representing~$f$. Let now $G$ be a polynomial representing $f$ and consider the polynomial ideal $I = \langle p(x_1),\ldots,p(x_n) \rangle \subseteq \mathbb C[x_1,\ldots,x_n]$ with $p(x) = \prod_{t \in T} (x - t)$. One can reduce $G$ modulo the ideal $I$ to get a polynomial $H$ of degree at most $m-1$ in each variable such that $G - H \in I$ and $\deg(G) \geq \deg(H)$. Since all the polynomials in $I$ vanish on $T^n$, we have that $H$ also represents $f$; thus $F = H$ and $\deg(G) \geq \deg(F)$.
\end{proof}

It is worth pointing out that there might be several polynomials of degree $\deg(f)$ representing $f$. However, only one of them can have degree at most $m-1$ in each variable. For example, consider the set $T = \{1,\varepsilon, \varepsilon^2\}$, with $\varepsilon$ a primitive third root of unity, and the $3$-ary function $f: T^3 \rightarrow T$ defined as $f(a,b,c) = a^2 b$ for all $a,b,c \in T$. Clearly, the polynomial $F(x,y,z) = x^2 y \in \mathbb C[x,y,z]$ represents $f$ and has degree at most $2$ in each variable and, by  \Cref{pr:uniquepoly}, $\deg(f) = \deg(F) = 3$. Also the polynomial $G(x,y,z) = x^2 y + z^3 - 1$ of degree $3$ represents $f$, but $G$ has degree $3$ in the variable $z$.

The other ingredient of the sensitivity conjecture is:

\begin{defi}[Sensitivity]\label{defi:sensitivity}
     The {\em local sensitivity} $s_{\mathbf{x}}(f)$ of an $m$-ary function $f$ at a vector $\mathbf{x}\in T^n$ is the number of elements $\mathbf{y}\in T^n$ which differ from $\mathbf{x}$ in exactly one entry and $f(\mathbf{x})\neq f(\mathbf{y})$.
     The {\em sensitivity} of $f$ is $s(f) = \max_{\mathbf{x}\in T^n}\{s_{\mathbf{x}}(f)\}$.
\end{defi}

Consider $T = \{t_0,\ldots,t_{m-1}\}$ a set with $m$ elements, $f:T^n \rightarrow T$ an $m$-ary function, and $\mathbf{x} = (t_{i_1},\ldots,t_{i_n}) \in T^n$. For every multisubset $S$ of $[n] := \{1,\ldots,n\}$, we denote by $\mult_S(j)$ the multiplicity of the element $j$ in $S$ and by $\X^S$ the vector obtained from $\X$ by replacing its $j$-th entry $t_{i_j}$ with $t_{i_j + \mult_S(j)}\in T$ for all $1\leq j\leq n$. Here the indices are taken modulo $m$.

For example, let $T = \{0,1,2\}$, $n=3$, and $\X_1=(0,0,0),\ \X_2=(1,0,2) \in T^3$.
If $S_1=\{1,1,2\}$, then ${\mult}_{S_1}(1) = 2,\ {\mult}_{S_1}(2) = 1,\ {\mult}_{S_1}(3) = 0$, and hence $\X_1^{S_1}= (2,1,0)$ and $ \X_2^{S_1} = (0,1,2)$.
If $S_2=\{1\}$, then $\X_1^{S_2}=(1,0,0)$ and $\X_2^{S_2}=(2,0,2)$.

We say that an element $t$ belongs to a multiset $A$, and write $t \in A$, whenever ${\mult}_A(t) > 0$.
We say that two multisubsets $A$ and $B$ of $[n]$ are disjoint if they do not contain common elements, regardless of their multiplicities. In the previous example, $S_1=\{1,1,2\}$ and $S_2=\{1\}$ are not disjoint, since both contain the element $1$.

\begin{defi}[Block sensitivity]\label{defi:block_sensitivity}
    The {\em local block sensitivity} $bs_{\X}(f)$ of an $m$-ary function $f$ at $\X\in T^n$ is the maximum $k$ for which there exist $k$ pairwise disjoint multisubsets $B_1,\dots,B_k$ of $[n]$ (the {\em sensitive blocks}) such that $f(\X)\neq f(\X^{B_i})$ for all $i\in\{1,\dots,k\}$. 
    The {\em block sensitivity} of $f$ is $bs(f)=\max_{\X\in T^n}\{bs_{\X}(f)\}$.
\end{defi}

\begin{rem}\label{rm:blockgeqsensit} Definitions~\ref{defi:sensitivity} and~\ref{defi:block_sensitivity} imply that $0 \leq \frac{s(f)}{ m-1 } \leq bs(f) \leq n$. This is because it is possible to change every entry of a vector $\X\in T^n$ in $m-1$ different ways, but only one of them can contribute to the computation of the block sensitivity (because the sensitive blocks are pairwise disjoint).
\end{rem}

\begin{ex}

    Let $T=\{0,1,2\}$, $n=2$, and $f:T^n\rightarrow T$ the $3$-ary function given by \begin{eqnarray*}
        f^{-1}(0)&=&\{(0,1)\}\, ,\\
        f^{-1}(1)&=&\{(0,2),\,(1,2),\,(2,2)\}\, ,\\
        f^{-1}(2)&=&\{(0,0),\,(1,0),\,(1,1),\,(2,0),\,(2,1)\}\, .
    \end{eqnarray*}

To compute the degree of $f$, following  \Cref{pr:uniquepoly}, we are going to find a polynomial $F \in \mathbb C[x_1,x_2]$ of degree at most $2$ in each variable which coincides with $f$ at every $\X\in\{0,1,2\}^2$. We can write $F$ as $$F(x_1,x_2)=a_{00}+a_{10}x_1+a_{11}x_1x_2+a_{01}x_2+a_{20}x_1^2+a_{21}x_1^2x_2+a_{22}x_1^2x_2^2+a_{12}x_1x_2^2+a_{02}x_2^2\, ,$$
with $a_{ij} \in \mathbb C$. Imposing that $F$ represents $f$ leads us to the following system of linear equations:

\[
\left \{
\begin{array}{rrrrrrrrrrrrrrrrrrrl}
a_{00} &  &  & & & & & & & & & & & & & & & = & 2 & \\
a_{00} & + & & & & + & a_{01} & + & & & & & & & & + & a_{02} & = & 0 &\\
a_{00} & + &  & & & + & 2a_{01} & + & & & & & & & & + & 4a_{02} & = & 1 &\\
a_{00} & + & a_{10}& +& & & & +& a_{20}& & & & & & & & & =& 2&\\
a_{00} & + & 2a_{10}& +& & & & +& 4a_{20}& & & & & & & & & =& 2&\\

a_{00} & + & a_{10}& +& 2a_{11}& +& 2a_{01}& +& a_{20}& +& 2a_{21}& +& 4a_{22}& +& 4a_{12}& +& 4a_{02}& =& 1& \\
a_{00} & + & 2a_{10}& +& 4a_{11}& +& 2a_{01}& +& 4a_{20}& +& 8a_{21}& +& 16a_{22}& +& 8a_{12}& +& 4a_{02}& =& 1&\\

a_{00} & + & a_{10}& +& a_{11}& +& a_{01}& +& a_{20}& +& a_{21}& +& a_{22}& +& a_{12}& +& a_{02}& =& 2&\\

a_{00} & + & 2a_{10}& +& 2a_{11}& +& a_{01}& +& 4a_{20}& +& 4a_{21}& +& 4a_{22}& +& 2a_{12}& +& a_{02}& =& 2& . 
\end{array}
\right .
\]

Solving the system we get $$F(x_1,x_2)=2+6x_1x_2-\frac{7}{2}x_2-2x_1^2x_2+x_1^2x_2^2-3x_1x_2^2+\frac{3}{2}x_2^2\ ,$$ and hence $\deg(f)=4$.

By \Cref{rm:blockgeqsensit} one has that $s(f) \leq 4$. If we consider now $\X=(0,1)$, then \[ s_{\X}(f)=|\{(0,0),(0,2),(1,1),(2,1)\}|=4,\] and this implies that $s(f)=4$.

Finally, we compute the block sensitivity of $f$. By \Cref{rm:blockgeqsensit}, $bs(f) \leq 2$. If we consider $\X_1=(1,0)$, then we cannot have a sensitive block consisting only of $1$. Consequently, every sensitive block must contain $2$ and this yields $bs_{\X_1}(f)=1$ taking for example $B=\{2,2\}$ as the sensitive block. In contrast, if we consider $\X_2=(0,1)$, then $B_1=\{1\}$ and $B_2=\{2\}$ are two disjoint sensitive blocks, and $bs_{\X_2}(f)=2$. This allows us to conclude that $bs(f)=2$.
\end{ex}

As mentioned in the introduction, the degree, the sensitivity and the block sensitivity are polynomially equivalent for Boolean functions. Before studying such properties for $m$-ary functions, a comment on the choice of the set $T\subseteq \mathbb{C}$ is in order.  When working with Boolean functions, the set $T$ is traditionally $T=\{0,1\}$,
  although another commonly used representation is $T = \{-1,1\}.$ 
  Most complexity measures of Boolean functions (including degree, sensitivity and block sensitivity) do not depend on the choice of the two-element-set $T\subseteq\mathbb C$. 
  If one neglects constant factors,  the same holds for general $m$-ary functions. 
  
  \begin{prop}\label{pr:changeofT}
    Let $T= \{t_0,\ldots,t_{m-1}\},T'= \{t_0',\ldots,t_{m-1}'\}\subseteq \mathbb{C}$ sets of size $m$. If $f: T^n \rightarrow T$ is an $m$-ary function, then $g: (T')^n \rightarrow T'$ obtained from $f$ by identifying $t_i$ with $t_i'$ is an $m$-ary function and:
    \begin{itemize}
        \item $s(f) = s(g)$,
        \item $bs(f) = bs(g)$,
        \item $(m-1)^{-2}\deg(f)\leq \deg(g) \leq (m-1)^2 \deg(f)$.
    \end{itemize}
    
    \end{prop}
  \begin{proof} It is clear that $s(f) = s(g)$ and $bs(f) = bs(g)$. To see the claim about the degree, let $F$ and $G$ be polynomials representing $f$ and $g$ with $\deg(f) = \deg(F)$ and $\deg(g) = \deg(G)$, respectively. Denote by $p,\,  q\in \mathbb C[x]$ the unique univariate polynomials of degree at most $m-1$ such that $p(t_i) = t_i'$ and $q(t_i') = t_i$ for all $i$. Then \[ F(x_1,\ldots,x_n) = q(G(p(x_1),\ldots,p(x_n))) \text{ for all } (x_1,\ldots,x_n) \in T^n,\] and $\deg(f) = \deg(F)  \leq(m-1)^2 \deg(G) = (m-1)^2 \deg(g)$. 
Similarly, one has that 
\[ G(x_1,\ldots,x_n) = p(F(q(x_1),\ldots,q(x_n))) \text{ for all } (x_1,\ldots,x_n) \in (T')^n,\] 
and $\deg(g) = \deg(G) \leq(m-1)^2 \deg(F)  = (m-1)^2 \deg(f)$. 
\end{proof}
As a consequence, when changing the set $T$ the degree can only change by a constant factor, and hence, for proving equivalence between complexity measures, one can consider any set $T\subseteq\mathbb{C}$ of cardinality $m$.

\section{An upper bound for the sensitivity in terms of the degree}\label{upperbound}

The goal of this section is to provide an upper bound for the sensitivity of an $m$-ary function in terms of its degree.  This is achieved in \Cref{th:cotafacil}, which follows from generalizing~\cite[Lemma 7]{NS94}, a result for Boolean functions, to the $m$-ary case:

\begin{lem}\label{lem:deg_bs}
    Let $T=\{t_0,t_1,\dots,t_{m-1}\} \subseteq \mathbb C$ and let $f:T^n\rightarrow T$ be an  $m$-ary function. Then, $$(m-1)\, \deg(f)\geq\sqrt{\frac{bs(f)}{2}}\, .$$
\end{lem}
\begin{proof}
    Denote $k := bs(f)$ and take $\T=(t_{i_1},t_{i_2},\dots,t_{i_n}) \in T^n$  such that $bs(f)=bs_{\T}(f)$.  We assume without loss of generality that $f(\T)=t_0$. Consider also $S_1,\dots,S_k$ a collection of pairwise disjoint multisubsets of $[n]$ such that $f(\T)\neq f(\T^{S_i})$ for all $i\in\{1,\dots,k\}$.  
    
    We define $g:\{0,1\}^k\rightarrow\{0,1\}$ as follows: we consider the polynomial $P \in \mathbb C[x]$ of degree $m-1$ such that $P(t_0) = 0$ and $P(t_i) = 1$ for $i \in \{1,\ldots,m-1\}.$ For $\Y := (y_1,\dots,y_k)\in\{0,1\}^k$ we take the multiset $S_\Y := \cup_{y_i = 1} S_i$ and define \[ g(\Y) := P( f(\T^{S_\Y}) )\, . \] 

The following facts hold:\begin{itemize}
    \item $g: \{0,1\}^k \rightarrow \{0,1\}$ is a Boolean function,

    \item $g(\textbf{0})=  P(f(\T)) = P(t_0) = 0$, and

     \item if $\mathbf{y} \in \{0,1\}^k$ has all its entries equal to $0$ expect the $i$-th one, which is equal to $1$, then $g(\mathbf{y})=  P( f(\T^{S_\Y}) ) = P( f(\T^{S_i}) ) = 1$. 
\end{itemize}

By~\cite[Lemma 5]{NS94}, we have that \begin{equation} \label{eq:cota} \deg(g) \geq\sqrt{\frac{k}{2}}=\sqrt{\frac{bs(f)}{2}}\, . \end{equation}

Moreover, if $F$ denotes a polynomial representing $f$ with $\deg(F) = \deg(f)$, and $\ell_1,\ldots,\ell_n \in \mathbb C[x_1,\ldots,x_k]$ are the unique linear forms such that $(\ell_1(\Y),\ldots,\ell_n(\Y)) =  \T^{S_\Y}$ for all $\Y \in \{0,1\}^k$,  then the polynomial $G = P(F(\ell_1,\ldots,\ell_n)) \in \mathbb C[x_1,\ldots,x_k]$ represents $g$. Hence, \begin{equation} \label{eq:cota2} \deg(g) \leq \deg(G) \leq \deg(P) \cdot  \deg(F) = (m-1) \deg(f).\end{equation} The result follows directly from (\ref{eq:cota}) and (\ref{eq:cota2}).   
\end{proof}

As a direct consequence of  \Cref{lem:deg_bs} together with  $bs(f)\geq\frac{s(f)}{m-1}$ (see \Cref{rm:blockgeqsensit}), we have the following theorem:

\thcotafacil*

We conjecture a polynomial relation into the other direction:
\conjsensit*

As in the Boolean case, graph theory might play an important role for proving the $m$-ary sensitivity conjecture. In the next section we present a  generalization of the Boolean equivalence theorem by Gotsman and Linial, which allows us to reformulate the sensitivity conjecture for $m$-ary functions in graph theoretical terms.

\section{The equivalence theorem for $m$-ary functions}\label{sec:equivalence}

All graphs we consider are simple, undirected and finite. 
For a graph $G = (V,E)$ and a vertex $v \in V$, we denote by $\deg_G(v)$ the degree of $v$ in $G$. We denote by $\Delta(G)$ (resp.  $\delta(G)$) the maximum (resp. minimum) degree of $G$. The maximum (resp. minimum) degree of the graph without vertices is $-\infty$ (resp. $+\infty$).

Huang's proof of the sensitivity conjecture for Boolean functions heavily relies on the equivalence theorem by Gotsman and Linial~\cite{GL92}. In this section we present an equivalence theorem for the $m$-ary case, which generalizes the Boolean one. Before presenting its statement we introduce some definitions, notations and some preliminary results that will be used in the proof.

For convenience, from now on we consider $T =  \Um := \roots$ the set of $m$-th roots of unity, where $\varepsilon$ is an $m$-th primitive root of unity. In the case of Boolean functions, the set $\{0,1\}^n$ can be seen as the vertex set of the $n$-dimensional hypercube $Q_n$. In the context of $m$-ary functions, a natural generalization is the \emph{Hamming graph} $H(n,m)$ whose vertex set is $\Umn$ (or $[ 0, m-1 ]^n := \{0,1,\ldots,m-1\}^n$) and where two vertices are adjacent if and only if they differ in exactly one of their entries. This graph can be seen as the $n$-th Cartesian power of the complete graph on $m$ vertices $K_m$. Some Hamming graphs are shown in \Cref{fig:cubos}.
It is easy to see that the Hamming graph $H(n,m)$ is $(m-1)n$-regular and the product of the entries of every vertex provides a proper $m$-coloring of it. We denote each set of the resulting $m$-partition by $C_{\varepsilon^k}$ for $k \in [ 0, m-1]$, i.e., $$C_{\varepsilon^k}=\left\{\mathbf{x}=(x_1,\hdots,x_n)\in\Umn \ \left| \ \prod_{j=1}^{n}x_j = \varepsilon^k \right. \right\}\ .$$

\begin{figure}[!h]
  \begin{center}
    \subfigure{
        \includegraphics[height=13cm]{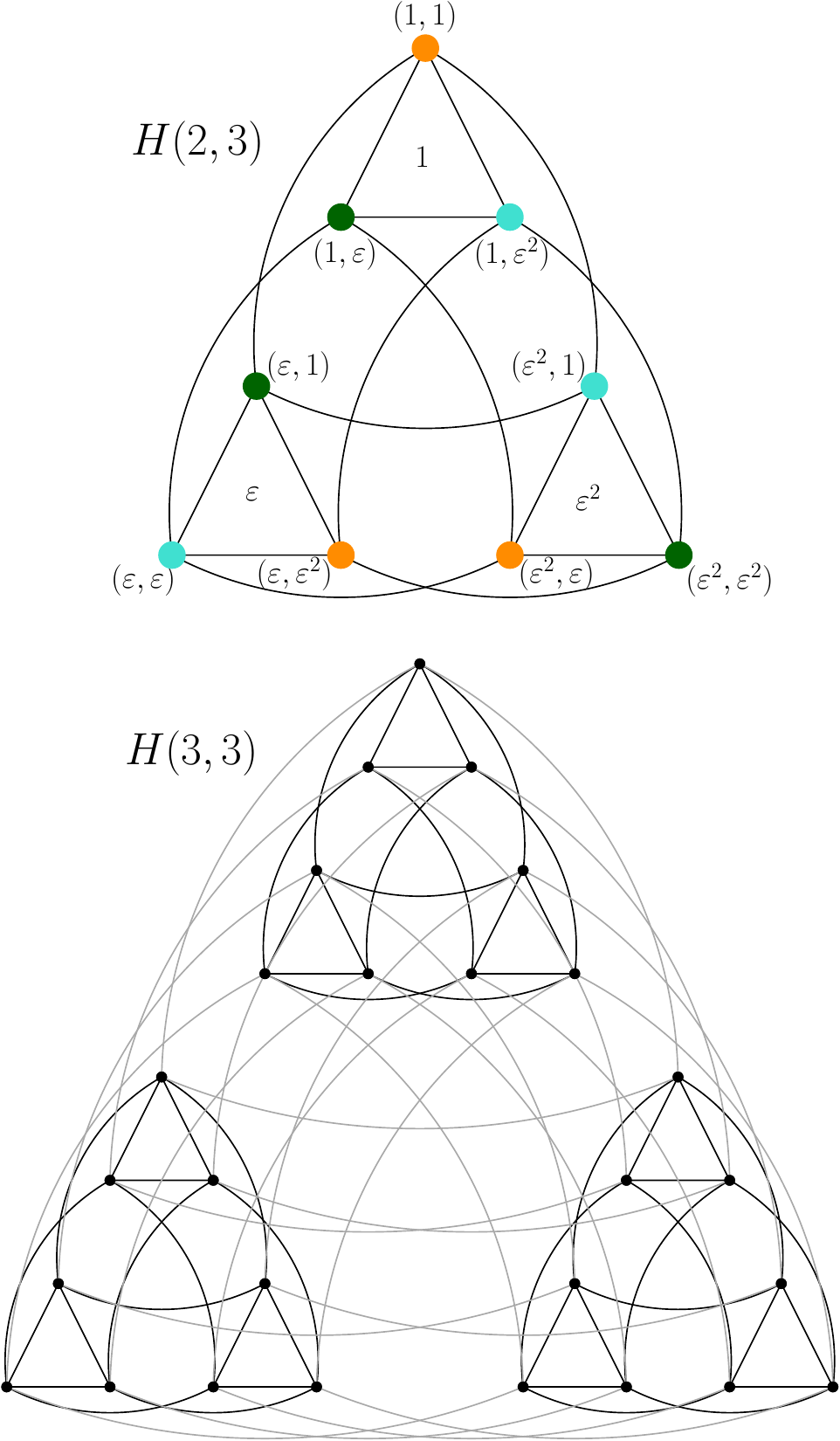}}
    \subfigure{
        \includegraphics[height=13cm]{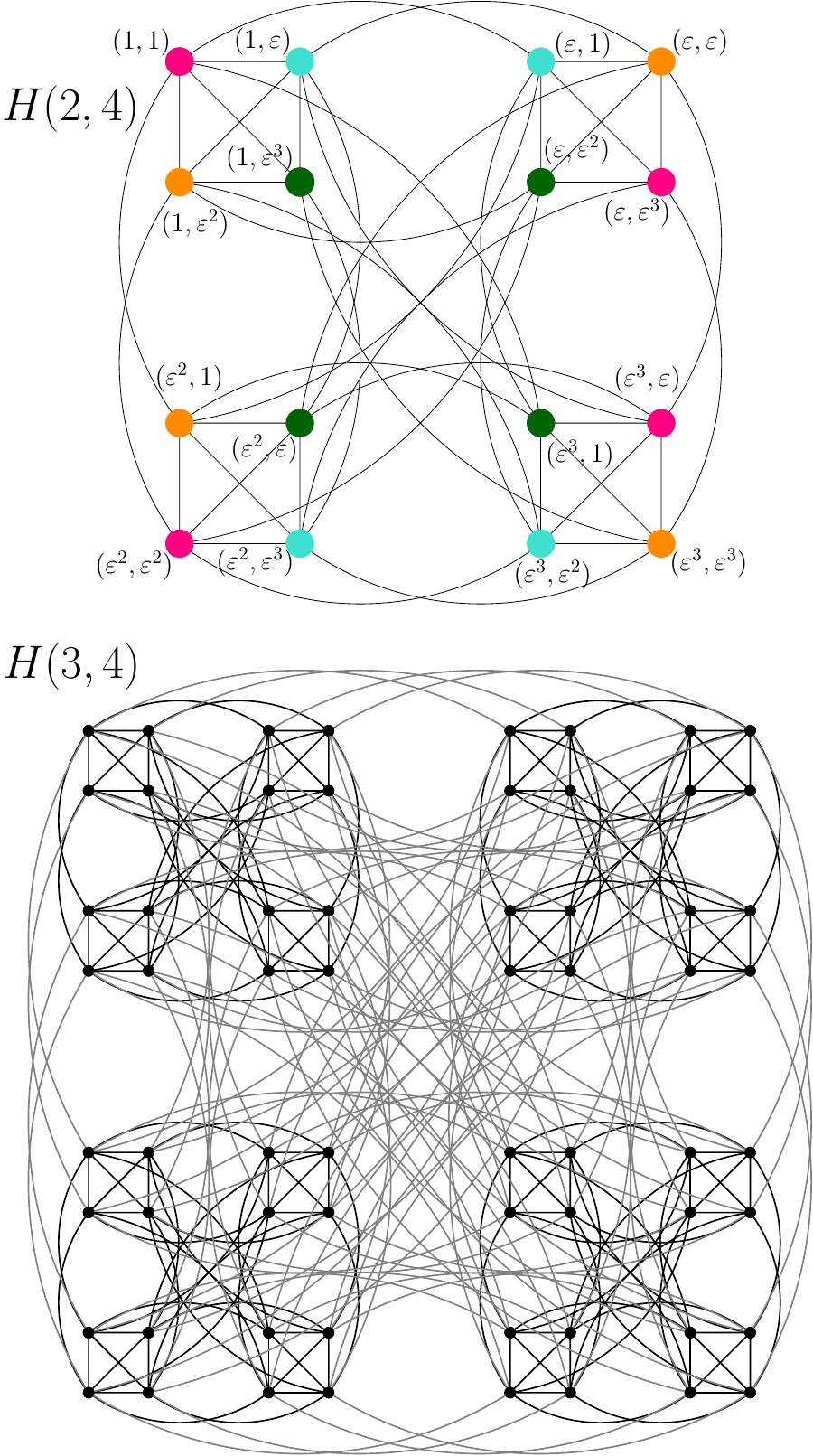}}
 \caption{Some Hamming graphs, the upper two with their proper $3$ and $4$-coloring, respectively}\label{fig:cubos}
  \end{center}
\end{figure}

One can associate to every $m$-ary function $f: \Umn \rightarrow \U_m$ a partition of the vertex set of $H(n,m)$ (where some parts might be empty) by just considering  $V_{\varepsilon^k}(f) := f^{-1}(\varepsilon^k)$ for all $k \in [ 0,m-1 ].$ If we denote by $H_{\varepsilon^k}$ the induced subgraph of $H(n,m)$ with vertex set $V_{\varepsilon^k}(f)$ for all $k \in [ 0, m-1 ]$, then for every vertex $\X \in V_{\varepsilon^k}(f)$ one has that $s_{\X}(f)=(m-1)n-\deg_{H_{\varepsilon^k}}(\X)$. Consequently, \begin{equation} \label{eq:partition} s(f)=(m-1)n-\min\{ \delta(H_{\varepsilon^k}) \, \vert \, 0 \leq k \leq m-1\}.\end{equation} Recall that the minimum degree of the empty graph is $+\infty$.

\begin{ex}\label{ex:running}
\Cref{fig:ejemplo1}  shows an example of three induced subgraphs $H_1, H_{\varepsilon}, H_{\varepsilon^2}$ whose vertices partition $\U_3^{\,3}$. We observe that $\delta(H_1) = \delta(H_{\varepsilon}) = \delta(H_{\varepsilon^2}) = 3$. If we consider the $3$-ary function $f: \U_{3}^{\,3} \rightarrow \U_3$ such that $f^{-1}(\varepsilon^j) = V(H_{\varepsilon^j})$, then $s(f) = 2 \cdot 3 - 3 = 3.$

 \begin{figure}[!h]
 \centering
 	\includegraphics[width=7cm]{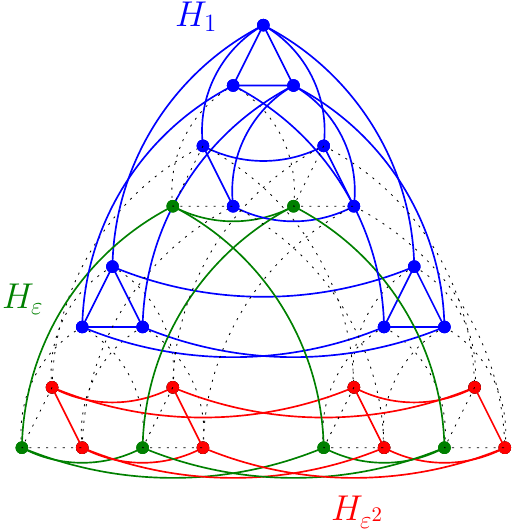}
 	\caption{Three induced subgraphs of $H(3,3)$ whose vertices partition $\U_{3}^3$}
 	\label{fig:ejemplo1}
 \end{figure}
\end{ex}

Let $f: \Umn \rightarrow \U_m$ be an $m$-ary function. For every $i\in [ 0, m-1 ]$, denote by $f_{\varepsilon^i}$ the $m$-ary function given by $f_{\varepsilon^i}(\X)=f(\X)\left(\prod_{j=1}^n x_j^i\right)$ for all $\X=(x_1,\dots,x_n)\in\U_m^{\,n}$. The following result relates the vertex partitions of $H(n,m)$ associated to $f$ and $f_\varepsilon$.

\begin{lem}\label{lem:vertices_rotaciones} Let $f: \Umn \rightarrow \Um$ be an $m$-ary function. 
For all $k \in [ 0, m-1 ]$ one has: \[  f_{\varepsilon}^{-1}(\varepsilon^k)  =\bigcup_{j\,\in\,[ 0, m-1 ]}(C_{\varepsilon^j} \cap f^{-1}(\varepsilon^{k-j})).\]
\end{lem}
\begin{proof}For all $\X = (x_1,\ldots,x_n) \in \U_m^{\,n}$ we have that $f_{\varepsilon}(\X)  = f(\X) \cdot (x_1 \cdots x_n).$
Taking $\varepsilon^j := x_1 \cdots x_n \in~\mathcal U_m$, we have that $\X \in C_{\varepsilon^j}$, and $f_{\varepsilon}(\X) = \varepsilon^k$ if and only if $f(\X) = \varepsilon^{k-j}\in\U_m$. 
\end{proof}

\begin{defi} Let $\{V_{\varepsilon^k} \, \vert \,  k \in [ 0,m-1] \}$ be a partition of $\U_m^{\,n}$ into $m$ subsets (where some of them might be empty). We denote by $\{\rho(V_{\varepsilon^k}) \, \vert \, k \in [ 0, m-1 ]\}$ the new partition of $\U_m^{\,n}$ given by
\[ \rho(V_{\varepsilon^k}) :=  \bigcup_{j\,\in\,[ 0, m-1 ]}(C_{\varepsilon^j}\cap V_{\varepsilon^{k-j}}) \  \text{ for all } k \in [ 0, m-1 ]. \]
\end{defi} 

\Cref{fig:rotation} illustrates how the sets $\rho(V_{\varepsilon^k})$ can be visualized as a certain cyclic rotation of the vertices of $V_{\varepsilon^k}$ through the $m$-partition given by the $C_{\varepsilon^j}$'s.

 \begin{figure}[!h]
 \centering
 	\includegraphics[width=10cm]{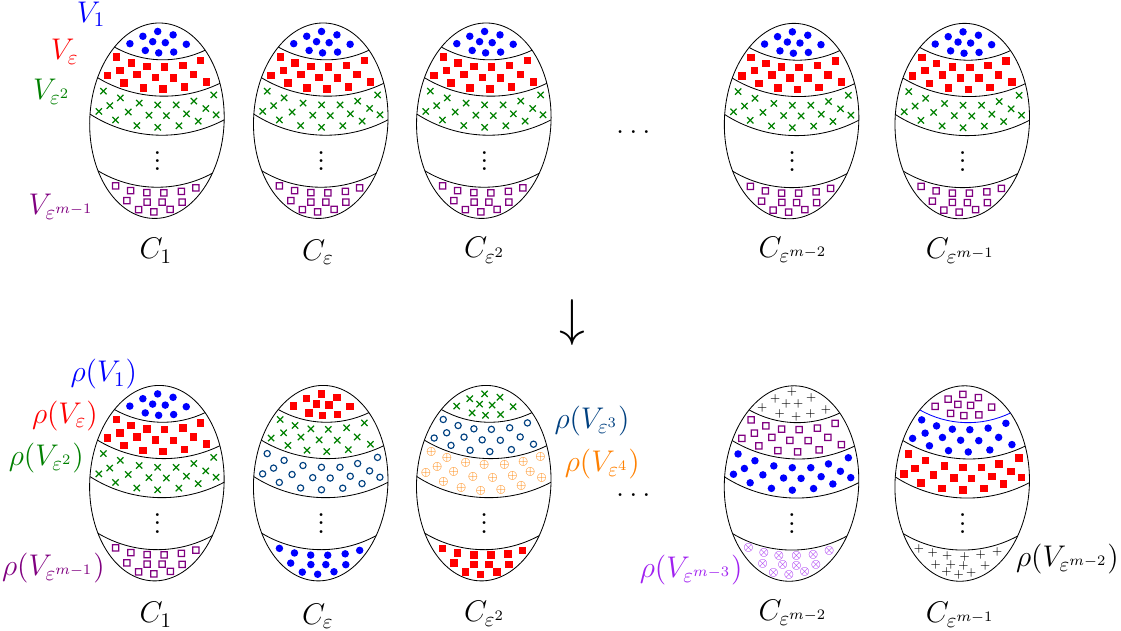}
 	\caption{Vertices of $\rho(V_{\varepsilon^k})$ as a certain cyclic rotation of the vertices of $V_{\varepsilon^k}$ through the partition given by the $C_{\varepsilon^j}$'s}
 	\label{fig:rotation}
 \end{figure}

For an $m$-ary function $f: \Umn \rightarrow \Um$, we consider the partition of $H(n,m)$ with $V_{\varepsilon^k}(f) = f^{-1}(\varepsilon^k)$ for all $k \in [ 0, m-1 ]$. By  \Cref{lem:vertices_rotaciones}, we have that $\rho(V_
{\varepsilon^k}(f))) = f_{\varepsilon}^{-1}(\varepsilon^k)$ for all $k \in [ 0, m-1 ]$. Hence, $\{\rho(V_{\varepsilon^k}(f)) \, \vert \, k \in [ 0, m-1 ] \}$ is the partition of $H(n,m)$ associated to $f_\varepsilon$.

\begin{ex}\label{ex:running2} If we consider the $3$-ary function $f: \U_{3}^{\,3} \rightarrow \U_3$ described in  \Cref{ex:running}, then  \Cref{fig:ejemplo2} shows the partition of $H(3,3)$ associated to $f_\varepsilon$. It turns out that 
$|f_\varepsilon^{-1}(1)| = |\rho(V(H_1))| = 8,\ |f_\varepsilon^{-1}(\varepsilon)| = |\rho(V(H_{\varepsilon}))| = 10$ and $|f_\varepsilon^{-1}(\varepsilon^2)| = |\rho(V(H_{\varepsilon^2}))| = 9.$

 \begin{figure}[!h]
 \centering
 	\includegraphics[width=15cm]{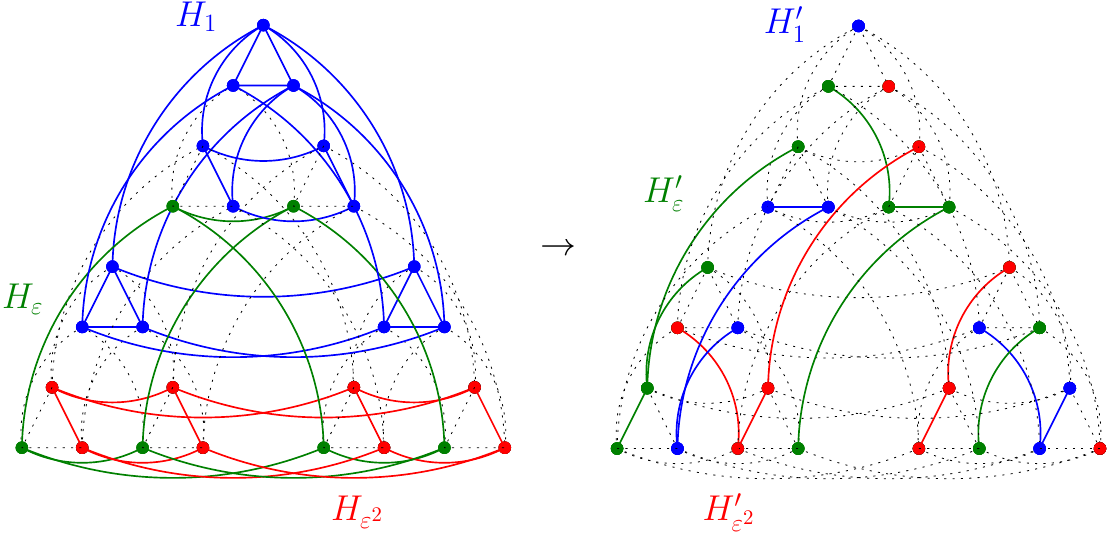}
 	\caption{Partition associated to $f_\varepsilon$ for the $3$-ary function $f$ of  \Cref{ex:running}, where $H_i'$ is the induced subgraph of $H(3,3)$ on $\rho(V(H_i))$}
 	\label{fig:ejemplo2}
 \end{figure}
\end{ex}

\begin{rem} \label{rm:additivenotation} Note that if we consider $[ 0, m-1 ]^n$ as the vertex set of $H(n,m)$, then a proper $m$-coloring of $H(n,m)$ is given by  $$D_k=\left\{\mathbf{x}=(x_1,\hdots,x_n)\in [ 0, m-1 ]^n \ \left| \ \sum_{j=1}^{n}x_j \equiv k  \mod m \right. \right\},$$
for $k \in [ 0, m-1 ]$.

Moreover, for a given  $m$-ary function $g: [ 0, m-1 ]^n \rightarrow [ 0, m-1 ]$ and for every $i\in [ 0, m-1 ]$, we denote by $g_{i}$ the $m$-ary function given by $g_{_i}(\X)=g(\X) + i(\sum_{j = 1}^n x_j) \mod m$ for all $\X=(x_1,\dots,x_n)\in [ 0, m-1 ]^n $. For each $k \in [ 0, m-1 ]$ we take $V_k(g) := g^{-1}(k)$ and denote 
\[ \rho(V_k(g))  = \bigcup_{j\,\in\, [ 0, m-1 ]} \left(D_j\cap V_{ ((k-j)\ {\rm mod}\ m) }(g) \right). \]
Then, $\rho(V_k(g)) = g_1^{-1}(k)$ for all $k \in [ 0, m-1 ]$.
\end{rem}

In the proof of the equivalence theorem we use the following lemma, which explicitly describes the constant term of the unique polynomial  of degree at most $m-1$ in each variable representing an $m$-ary function $f: \Umn \rightarrow \Um$. This result depends on the fact that $T = \Um$ and does not hold for other choices of $T$.

\begin{lem}\label{lem:esperanza}
    If $f:\Umn\rightarrow\Um$ is an $m$-ary function and $F\in\mathbb C[x_1,\dots,x_n]$ is the unique polynomial of degree at most $m-1$ in each variable representing it, then $F(\mathbf{0})=E(f)$, where $E(f)=\frac{1}{m^n}\sum_{\mathbf{x}\in\U_m^n}f(\mathbf{x})$ is the average value of $f$ on $\Umn$.
\end{lem}
\begin{proof}
    We denote by $\mathcal{P}([n])_{\leq(m-1)}$ the set of all multisubsets of $[n]$ in which every element has multiplicity at most $m-1$. Since $F$ has degree at most $m-1$ in each variable, it can be written as
 $F(x_1,\dots,x_n) = \sum_{I\in \mathcal{P}([n])_{\leq(m-1)}}\hat{f}(I)\ \prod_{i\in I}x_i^{\mult_I(i)}$, where $\hat{f}(I) \in \mathbb C$ for all $I\in \mathcal{P}([n])_{\leq(m-1)}$, and we have that \begin{eqnarray*}
    E(f) & = & \frac{1}{m^n}\sum_{\X\in \U_m^n}f(\X)\ =\ \frac{1}{m^n}\sum_{\X\in \U_m^n}F(\X) \ = \\  
    & = &\frac{1}{m^n}\sum_{\X\in \U_m^n}\left[ \sum_{I\in\mathcal{P}([n])_{\leq(m-1)}}\hat{f}(I)\prod_{i\in I}x_i^{\mult_I(i)}\right] \ = \\
    & = & \frac{1}{m^n}\left[\sum_{\X\in \U_m^n}\hat{f}(\emptyset)+\sum_{\X\in \U_m^n}\sum_{\substack{I\in\mathcal{P}([n])_{\leq(m-1)}\\I\neq\emptyset}} \hat{f}(I)\prod_{i\in I}x_i^{\mult_I(i)}\right]\ = \\
    & = & \frac{1}{m^n}\left[m^n\ F(\mathbf{0})+\sum_{\substack{I\in\mathcal{P}([n])_{\leq(m-1)}\\I\neq\emptyset}} \hat{f}(I)\sum_{\X\in \U_m^n}\prod_{i\in I}x_i^{\mult_I(i)}\right]\ ,
\end{eqnarray*} so it is enough to prove that $\sum_{\X\in \U_m^n}\prod_{i\in I}x_i^{\mult_I(i)}=0$ for every nonempty $I\in\mathcal{P}([n])_{\leq(m-1)}$.

Take $I\neq\emptyset$, assume without loss of generality that $1\in I$, and take $M:=\mult_I(1)$. Then \begin{eqnarray*}
    \sum_{\X\in \U_m^n}\prod_{i\in I}x_i^{\mult_I(i)}&=&\sum_{\ell=0}^{m-1}\left(\sum_{(\varepsilon^{\ell},x_2,\dots,x_{n})\in \U_m^n}\varepsilon^{\ell M}\prod_{i\in I\setminus\{1\}}x_i^{\mult_I(i)}\right)\ =\\
    &=&\sum_{\Y=(x_2,\dots,x_{n})\in \U_m^{n-1}}\left(\prod_{i\in I\setminus\{1\}}x_i^{\mult_I(i)}\right)\left(\sum_{\ell=0}^{m-1}\varepsilon^{\ell M} \right)\ =\ 0,
\end{eqnarray*}
where the last equality follows from the fact that $\sum_{\ell=0}^{m-1} \mu^{\ell}=0$ for every $\mu \in \Um$. Thus, 
we conclude the result.
\end{proof}

\begin{lem}\label{lm:coefrotacion} Let $f: \Umn \rightarrow \Um$ be an $m$-ary function 
and denote by $F, G \in \mathbb C[x_1,\dots,x_n]$ the unique polynomials of degree at most $m-1$ in each variable representing $f$ and $f_{\varepsilon^{m-1}}$, respectively. Then, $F(\mathbf{0})$ coincides with the coefficient of $(x_1 \cdots x_{n})^{m-1}$ in $G$. In particular, $\deg(f_{\varepsilon^{m-1}}) = (m-1)n$ if and only if $E(f) \neq 0.$
\end{lem}
\begin{proof}
Clearly $H := (x_1 \cdots x_n)^{m-1} F$ is a polynomial representing $f_{\varepsilon^{m-1}}$, and hence $G$ is the reduction of $H$ modulo the ideal $I  = \langle x_1^m - 1,\ldots,x_n^m - 1 \rangle$. The result follows from the fact that reducing a polynomial modulo $I$ consists of just taking all the exponents appearing in the expression modulo $m$.  In particular, $\deg(f_{\varepsilon^{m-1}}) = (m-1)n$ if and only if $F(\mathbf{0}) \neq 0$, and the latter is equivalent to  $E(f) \neq 0$ by  \Cref{lem:esperanza}.
\end{proof}

The next result shows how to graph theoretically construct an $m$-ary function $f: \U_m^{\,n}  \rightarrow \U_m$ of degree $\deg(f) = (m-1)n$ (the maximum possible degree) with prescribed sensitivity. This result will be particularly useful in the next section.

\begin{prop} \label{pr:grafosafunciones} Let $m \geq 2$ and $n, s \geq 1$. The following are equivalent:
\begin{itemize}
\item[$(1)$] There is an $m$-ary function $f: \U_m^{\,n} \rightarrow \U_m$ of degree $\deg(f) = (m-1)n$ and sensitivity $s(f) = s$.

\item[$(2)$] There are (possibly empty) induced subgraphs $\{H_{\varepsilon^k}\, \vert \, 0 \leq k \leq m-1\}$ of $H(n,m)$ such that their vertex sets partition $V(H(n,m))$, $\sum_{k=0}^{m-1}|\rho(V(H_{\varepsilon^k}))|\ \varepsilon^k\neq 0,$ and $$(m-1)n - \min\{ \delta(H_{\varepsilon^k}) \, \vert \, 0 \leq k \leq m-1\} = s\ .$$
\end{itemize}
\end{prop}
\begin{proof}For an $m$-ary function $f: \Umn \rightarrow \Um$, one can consider the family of induced subgraphs $\{H_{\varepsilon^k} \, \vert \, 0 \leq k \leq m-1\}$ of $H(n,m)$ such that  $V(H_{\varepsilon^k}) = f^{-1}(\varepsilon^k)$. Conversely, for any induced subgraphs $\{H_{\varepsilon^k}\, \vert \, 0 \leq k \leq m-1\}$ of $H(n,m)$ such that their vertex sets partition $V(H(n,m))$, one can consider the $m$-ary function $f: \Umn \rightarrow \Um$ defined as $f(\X) = \varepsilon^k$ if and only if $\X \in V(H_{\varepsilon^k}).$

In both cases, by Formula (\ref{eq:partition}), we have that \[ s(f)=(m-1)n-\min\{ \delta(H_{\varepsilon^k}) \, \vert \, 0 \leq k \leq m-1\}. \]
Also, taking $g = f_\varepsilon$ we have that $f = (f_\varepsilon)_{\varepsilon^{m-1}} = g_{\varepsilon^{m-1}}$. By  \Cref{lm:coefrotacion}, it follows that $\deg(f) = (m-1)n$ if and only if $E(g) \neq 0$. Finally, by  \Cref{lem:vertices_rotaciones}, we have that $E(g) = \frac{1}{m^n}\sum_{k=0}^{m-1}|\rho(V(H_{\varepsilon^k}))|\ \varepsilon^k,$ and the result follows.
\end{proof}

Now we can prove the equivalence theorem.

\begin{thm}[\bf Equivalence theorem for $m$-ary functions]\label{thm:equivalence}
    Let $h:\mathbb{N}\rightarrow\mathbb{R}$ be a function and $m$ a positive integer. The following are equivalent:
    \begin{itemize}
        \item[$(1)$] For any collection of (possibly empty) induced subgraphs $\{H_{\varepsilon^k}\, \vert \, 0 \leq k \leq m-1\}$ of $H(n,m)$ whose vertex sets partition $V(H(n,m))$ and 
        $\sum_{k=0}^{m-1}|\rho(V(H_{\varepsilon^k}))|\ \varepsilon^k\neq 0$, it holds $$(m-1)n-\min\{ \delta(H_{\varepsilon^k}) \, \vert \, 0 \leq k \leq m-1 \} \geq h(n).$$
    
    \item[$(2)$] For any $m$-ary function $f:\Umn \rightarrow \Um$, $s(f)\geq h\left(\frac{1}{m-1}\, \deg(f)\right)$. 
    \end{itemize}
\end{thm}
\begin{proof}

First of all, we transform statement $(1)$ into the equivalent statement $(1')$ concerning $m$-ary functions:

\begin{itemize}
    \item[$(1')$] For any $m$-ary function $f:\Umn \rightarrow \Um$ with $E(f)\neq 0$, there is a vector $v\in\Umn$ such that $h(n)\leq s_v(f_{\varepsilon^{m-1}})$.
\end{itemize}

We begin by proving that $(1')\Rightarrow (1)$. Let $H_1,H_{\varepsilon},\dots,H_{\varepsilon^{m-1}}$ be $m$ induced subgraphs of $H(n,m)$ such that their vertex sets partition $\Umn$ and $\sum_{k=0}^{m-1}|\rho(V(H_{\varepsilon^k}))|\ \varepsilon^k\neq 0$.

We define the $m$-ary function $f:\Umn \rightarrow \Um$ given by $f^{-1}(\varepsilon^k)=V(H_{\varepsilon^k})$ for all $k\in[0,m-1]$. Then, by  \Cref{lem:vertices_rotaciones} we have that $f_{\varepsilon}^{-1}(\varepsilon^k) = \rho(V(H_{\varepsilon^k}))$ for all $k\in[0,m-1].$
 Moreover, \[ 0 \neq \sum_{k=0}^{m-1}|\rho(V(H_{\varepsilon^k}))|\ \varepsilon^k  = \sum_{\X\in \Umn}f_{\varepsilon}(\X)=m^n\ E(f_{\varepsilon}),\] being $E(f_{\varepsilon})$ the average value of $f_{\varepsilon}$ on $\Umn$. Then $E(f_{\varepsilon})\neq 0$ and, by $(1')$, there is a vector $v\in\Umn$ satisfying $h(n)\leq s_v((f_{\varepsilon})_{\varepsilon^{m-1}})=s_v(f)$. 

Now, by Equation (\ref{eq:partition}),  $$(m-1)n-\min\{ \delta(H_{\varepsilon^k}) \, \vert \, 0 \leq k \leq m-1\} = s(f) \geq s_v(f) \geq h(n).$$

    Now we prove that $(1)\Rightarrow(1')$. Let $f:\Umn \rightarrow\Um$ be an $m$-ary function with $E(f)\neq 0$. We define $H_{\varepsilon^k}$ as the induced subgraph of $H(n,m)$ on $(f_{\varepsilon^{m-1}})^{-1}(\varepsilon^k)$ for all $k\in[0,m-1]$. By  \Cref{lem:vertices_rotaciones}, $\rho(V(H_{\varepsilon^k}))=((f_{\varepsilon^{m-1}})_{\varepsilon})^{-1}(\varepsilon^k) = f^{-1}(\varepsilon^k)$ for all $k\in[0,m-1]$, and hence $\sum_{i=0}^{m-1}|\rho(V(H_{\varepsilon^i}))|\ \varepsilon^i = m^n E(f) \neq 0$.   

By $(1)$, we get that $$\min\{ \delta(H_{\varepsilon^k}) \, \vert \, 0 \leq k \leq m-1\} \leq (m-1)n-h(n)\ ,$$
so there must be some $\varepsilon^j \in \Um$ and a vertex $v\in V(H_{\varepsilon^j})$ such that $\deg_{H_{\varepsilon^j}}(v)\leq(m-1)n-h(n)$. Furthermore, since $V(H_{\varepsilon^j})=(f_{\varepsilon^{m-1}})^{-1}(\varepsilon^j)$ for all $j\in[0,m-1]$, we have that $\deg_{H_{\varepsilon^j}}(v) = (m-1)n-s_v(f_{\varepsilon^{m-1}})$. Hence $(m-1)n-s_v(f_{\varepsilon^{m-1}})\leq (m-1)n-h(n) $, and we conclude that $h(n)\leq s_v(f_{\varepsilon^{m-1}})$.

\medskip

We now rewrite statement $(2)$ in an equivalent way:
\begin{itemize}
    \item[$(2')$] For any $m$-ary function $f: \Umn \rightarrow \Um$, $s(f)<h(n)$ implies that $\deg(f)<(m-1)n$.
\end{itemize}

We start by showing that $(2)\Rightarrow(2')$. Let $f$ be an arbitrary $m$-ary function with $\deg(f) \geq(m-1)n$. Then we have that $\deg(f)=(m-1)n$ (because $\deg(f) \leq (m-1)n$). Let us prove that $s(f)\geq h(n)$. By $(2)$, we have that $s(f)\geq h\left(\frac{1}{m-1}\deg(f)\right)$, and thus $$s(f)\geq h\left(\frac{1}{m-1}\deg(f)\right)=h\left(\frac{1}{m-1}(m-1)n\right)=h(n)\ .$$

Now we prove that $(2')\Rightarrow(2)$ by contradiction. Assume that $(2')$ holds and that there exists an $m$-ary function $f$ with $s(f)<h\left(\frac{1}{m-1}\deg(f)\right)$. Then, by $(2')$, $\deg(f)<(m-1)\frac{1}{m-1}\deg(f)=\deg(f)$, a contradiction.

At this point, it is enough to show that statements $(1')$ and $(2')$ are equivalent for any function $h:\mathbb{N}\rightarrow\mathbb{R}$ and any natural number $m$.
\begin{itemize}
    \item[$(1')$] For any $m$-ary function $f:\Umn \rightarrow\Um$ with $E(f)\neq 0$, there is a vector $v\in\Umn$ satisfying $h(n)\leq s_v(f_{\varepsilon^{m-1}})$.
    \item[$(2')$] For any $m$-ary function $f:\Umn\rightarrow\Um$, $s(f)<h(n)$ implies that $\deg(f)<(m-1)n$.
\end{itemize}

We begin by proving $(2')\Rightarrow (1')$. Let us see that $E(f)=0$ for any $m$-ary function $f$ such that $h(n)>s_v(f_{\varepsilon^{m-1}})$ for all $v\in \Umn$. We have that $s(f_{\varepsilon^{m-1}})=\max_{v\in \Umn} s_v(f_{\varepsilon^{m-1}})<h(n)$, and by $(2')$ this implies that $\deg(f_{\varepsilon^{m-1}})<(m-1)n$. Then, by  \Cref{lm:coefrotacion} we are done.

It only remains to prove $(1')\Rightarrow(2')$. Since $\deg(f)\leq(m-1)n$ for every $m$-ary function $f$, $\deg(f)<(m-1)n$ is equivalent to $\deg(f)\neq(m-1)n$. Take an $m$-ary function $f$ with $\deg(f)=(m-1)n$, and let us prove that $s(f)\geq h(n)$. Taking $g := f_\varepsilon$, we have that $g_{\varepsilon^{m-1}} = f$. Hence, applying \Cref{lm:coefrotacion} to $g$, we have that $E(g) = E(f_{\varepsilon}) \neq 0$. Then, by $(1')$ there is a vector $v\in \Umn$ satisfying that $h(n)\leq s_v((f_{\varepsilon})_{\varepsilon^{m-1}})=s_v(f)$, and we directly conclude that $s(f)=\max_{\X\in \Umn}s_{\X}(f)\geq s_v(f)\geq h(n)$.\end{proof}

\begin{ex}The three induced subgraphs $H_1, H_{\varepsilon}, H_{\varepsilon^2}$ whose vertices partition $\U_3^{\,3}$ and were studied in Examples \ref{ex:running} and \ref{ex:running2} satisfy that $\sum_{k = 0}^2 |\rho(V(H_{\varepsilon^k}))|\varepsilon^k =  9 \varepsilon^2\ + 10 \varepsilon + 8 =  \varepsilon - 1 \neq 0$, and $2 \cdot 3 -\min\{\delta(H_1),\delta(H_{\varepsilon}),\delta(H_{\varepsilon^2})\} = 3$. As a consequence, if there exists $h: \mathbb N \rightarrow \mathbb R$ satisfying (any of) the conditions of  \Cref{thm:equivalence} for $3$-ary functions, then $h(3) \leq 3$.   
\end{ex}

\begin{rem}
    Although the equivalence theorem for $m$-ary functions looks different from its Boolean version, it is a generalization of it. 
    Indeed, if $m = 2$, then $\Um = \{-1,1\}$, and if one denotes $H_1'$ (respect. $H_{-1}'$) the induced subgraphs with vertices $\rho(V(H_1))$ (respect. $\rho(V(H_{-1}))$), in  \Cref{thm:equivalence} one has that 
    \[ n-\min\{\delta(H_1),\delta(H_{-1})\} = \max\{ \Delta (H_1'), \Delta(H_{-1}')\} \]    
    (see  \Cref{fig:Boolean_switch}), and the condition $\sum_{i=0}^{m-1}|\rho(V(H_{\varepsilon^i}))|\ \varepsilon^i\neq 0$ is translated to $|V(H_{1}')| \neq |V(H_{-1}')|$. Hence, for $m = 2$ and taking $G := H_1'$, condition (1) can be equivalently restated as:
    
    For any induced subgraph $G$ of $Q_n$ such that $|V(G)| \neq 2^{n-1},$ then $\max\{\Delta(G),\Delta(Q_n - G)\}\geq h(n),$ where $Q_n - G$ denotes the induced subgraph of $Q_n$ with $V(Q_n - G) = V(Q_n) \setminus V(G).$

\begin{figure}[!h]
  \begin{center}
    
        \includegraphics[height=3cm]{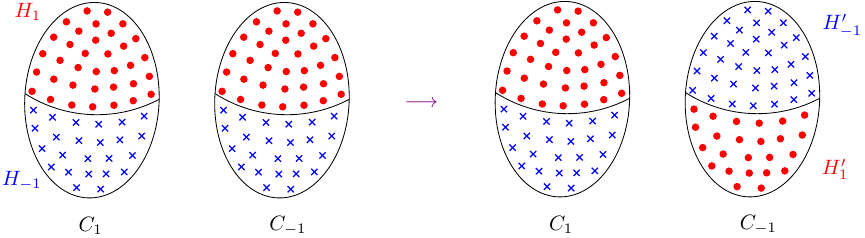}
\caption{ \Cref{thm:equivalence} generalizes Gotsman-Linial's Boolean equivalence theorem }
\label{fig:Boolean_switch}
  \end{center}
\end{figure}

 \end{rem}

The equivalence theorem allows us to translate the sensitivity conjecture for $m$-ary functions into an equivalent graph theoretical problem. 

\begin{conj}\label{conj:sensitivity_graphs}
    For every positive integer $m$ there exists $\mu>0$ such that for any collection of (possibly empty) induced subgraphs $H_1, H_{\varepsilon},\dots,H_{\varepsilon^{m-1}}$ of $H(n,m)$ whose vertex sets partition $\Umn$ and $\sum_{k=0}^{m-1}|\rho(V(H_{\varepsilon^k}))|\ \varepsilon^k\neq 0$, it holds $$(m-1)n-\min\{ \delta(H_{\varepsilon^k}) \, \vert \, 0 \leq k \leq m-1\} \in \Omega(n^\mu)\ .$$
\end{conj}

\begin{rem}\label{rm:primo} Note that when $m=p$ is a prime number, the $m$-th cyclotomic polynomial is $\Phi_p(x)=x^{p-1}+x^{p-2}+\dots+x+1\in\mathbb{Q}[x]$. Hence, 
$\sum_{i = 0}^{p-1} a_i \varepsilon^i = 0$ with $a_0,\ldots,a_{p-1} \in \mathbb Q$ if and only if $a_0 = \cdots = a_{p-1}$. As a consequence, when $m$  is a prime number, the condition ``$\sum_{k=0}^{m-1}|\rho(V(H_{\varepsilon^k}))|\ \varepsilon^k\neq 0$'' in the statement of \Cref{conj:sensitivity_graphs} can be replaced by the equivalent condition ``there exists $k \in [ 0, p-1 ]$ such that $|\rho(V(H_{\varepsilon^k}))| \neq p^{n-1}$''.
\end{rem}

In the next section, we prove that if $m$ is a prime number and \Cref{conj:sensitivity_graphs} holds for $m$-ary functions, then $\mu \leq 1/2$.

\section{Quadratic separation between the sensitivity and the degree}

Chung, Füredi, Graham and Seymour~\cite{CFGS88} provided an induced subgraph of the $n$-dimensional hypercube with $2^{n-1}+1$ vertices whose maximum degree is strictly smaller than $\sqrt{n}+1$. By Gotsman-Linial's equivalence theorem, this resulted in the existence of Boolean functions with $s(f) < \sqrt{\deg(f)} + 1$. Huang~\cite{Hua19} later proved that $s(f) \geq \sqrt{\deg(f)}$ for every Boolean function $f$.  The goal of this section it to extend the construction in~\cite{CFGS88} to the Hamming graph $H(n,m)$ and use the equivalence theorem for $m$-ary functions to produce $m$-ary functions with low sensitivity and high degree.  For convenience, in this section we work with the set $[0,m-1]$ instead of $\{1,\varepsilon,\dots,\varepsilon^{m-1}\}$; consequently, we denote the vertex set of $H(n,m)$ by $[0,m-1]^n$. 

Let $m, n \geq 2$. There is a natural bijection between the set $\mathcal{P}([n])_{\leq(m-1)}$ of multisubsets of $[n]$ with multiplicity at most $m-1$ and  $[ 0,m-1 ]^n$. This bijection sends
$I \in  \mathcal{P}([n])_{\leq(m-1)}$ to the $n$-tuple $(\mult_I(1),\ldots,\mult_I(n)) \in [ 0, m-1 ]^n$, and lets us identify the vertices of $H(n,m)$ with elements of $\mathcal{P}([n])_{\leq(m-1)}$. 

For $A \subset [n]$ and $j \geq 1$, we denote by $A^{(j)}$ the multisubset of $[n]$ with the same elements as $A$, all of them with multiplicity $j$. In $\mathcal{P}([n])_{\leq(m-1)}$ we consider the partial order given by inclusion, where $P\subseteq Q$ if the multiplicity of every element in $P$ is smaller than or equal to its multiplicity in $Q$.  For every multisubset $P \in \mathcal{P}([n])_{\leq(m-1)}$ we denote by $\uparrow \{P\}$ its filter, i.e., the set $\uparrow \{P\} = \{Q\in\mathcal{P}([n])_{\leq(m-1)} : P\subseteq Q\}$, and for $F \subseteq \mathcal{P}([n])_{\leq(m-1)}$ we denote  $\uparrow F = \cup_{P \in F} \uparrow \{P\}$.

\begin{lem}\label{lm:construcion}
Let $m,n \geq 2$, and let $\{A_1,\ldots,A_k\}$ be a partition of $[n]$ into $k$ nonempty sets. We denote $F_i := \uparrow\{A_1^{(i)},A_2^{(i)},\dots,A_k^{(i)}\}$ and consider the following induced subgraphs of  $H(n,m)$:\begin{itemize}
        \item $G_{m-1}$ is the induced subgraph on $F_{m-1}$,
        \item $G_i$ is the induced subgraph on $F_{i} \setminus F_{i+1}$ for $i\in\{1,2,\dots,m-2\}$, and
        \item $G_{0}$ is the induced subgraph on $\mathcal{P}([n])_{\leq(m-1)}\setminus F_1$.
    \end{itemize}
Then $\{V(G_i) \, \vert\,  0 \leq i \leq m-1\}$ constitute a partition of $V(H(n,m))$, and
\[ \delta(G_i) = (m-1)(n-k) + i (k  - \max(|A_1|,\ldots,|A_k|)) \text{ for  } 0 \leq i \leq m-1. \] In particular, $(m-1)n - \min\{\delta(G_0),\dots,\delta(G_{m-1})\} =  (m-1)\, \max(k, |A_1|, \ldots, |A_k|)$. 
\end{lem}

 \begin{proof}
It is clear that $V(G_0),V(G_1),\dots,V(G_{m-1})$ constitute a partition of $V(H(n,m))$. Let us compute  the values $\delta(G_0),\delta(G_1),\ldots,\delta(G_{m-1})$.

Take $B\in V(G_{m-1})$. Then there is at least one index $i\in[k]$ such that $A_i^{(m-1)}\subseteq B$. 
If there are two indices verifying the previous property, then all the neighbors of $B$ are in $G_{m-1}$ and $\deg_{G_{m-1}}(B) = (m-1)n$. If there is only one index $i_0$ such that $A_{i_0}^{(m-1)}\subseteq B$, then changing the multiplicity in $B$ of any element of $A_{i_0}$ (from $m-1$ to any value in  $[ 0, m-2 ]$) one gets a multiset not containing $A_{i_0}^{(m-1)}$ anymore, and hence $\deg_{G_{m-1}}(B) = (m-1)(n - |A_{i_0}|)$. As a consequence, $\delta(G_{m-1}) = (m-1)(n- \max(|A_1|,\ldots,|A_k|))$.

Take now $B\in V(G_{0})$. Then $B$ does not contain $A_i$ for any $i\in[k]$. If there are at least two elements in each $A_i$ which do not belong to $B$, then all the neighbors of $B$ are in $G_{0}$ and $\deg_{G_0}(B) = (m-1)n$. On the other hand, for each $i\in[k]$ such that $B$ does not contain exactly one element of $A_i$, then we can add to $B$ this missing element with all possible multiplicities to get $(m-1)$ vertices adjacent to $B$ and not in $G_{0}$. Since this can happen for all $i\in[k]$, $B$ has at most $(m-1)k$ neighbors not in $G_{0}$, and there are vertices attaining this bound. Hence, $\delta(G_{0})=(m-1)(n-k)$.

Now let $i\in\{1,\dots,m-2\}$ and $B\in V(G_i)$. Then there is some $j\in[k]$ such that $A_j^{(i)}\subseteq B$ but $A_s^{(i+1)}\not\subseteq B$ for all $s\in[k]$. Let us compute how many neighbors of $B$ are not in $G_i$.

First, we compute the number of neighbors of $B$ not in $F_i$. If $B$ contains more than one multiset $A_j^{(i)}$, then $B$ has all its  neighbors in $F_i$. Assume that there is only one index $j_0$ such that $A_{j_0}^{(i)}\subseteq B$. In order to obtain a neighbor of $B$ which does not contain $A_{j_0}^{(i)}$, we have to change the multiplicity in $B$ of any element of $A_{j_0}$ to any value in $[ 0, i-1 ]$. Hence $B$ has at most $i |A_{j_0}|$ neighbors not in  $F_i$.

Now we compute the number of vertices adjacent to $B$ in $F_{i+1}$. If we change the multiplicity of only one element of $B$ and as a result we obtain a new multiset $B'$ such that $A_s^{(i+1)}\subseteq B'$ for some $s\in[k]$, then all the elements of $A_s$ already have multiplicity at least $i+1$ in $B$ except one, which is the one whose multiplicity we have to change to any value in $[i+1, m-1]$ ($m-i-1$ different possibilities). Since this can happen with every $A_j$, we conclude that $B$ has at most $k(m-i-1)$ neighbors in $F_{i+1}$.
As a consequence, \[ \delta(G_i) \geq (m-1)n - i (\max(|A_1|,\ldots,|A_k|)) - k (m-1-i) \] for all $i \in [0,m-1],$ and there are vertices achieving this bound. Hence, $\delta(G_i) = (m-1)(n-k) + i (k  - \max(|A_1|,\ldots,|A_k|))$ and the result follows.\end{proof}

If one takes a partition $\{A_1,\ldots,A_k\}$ of $[n]$ such that $\max(k, |A_1|, \ldots, |A_k|) = \lceil \sqrt{n} \rceil$, applying the previous lemma one gets that $$(m-1)n-\min\{\delta(G_0),\delta(G_1),\dots,\delta(G_{m-1})\}=(m-1)\lceil \sqrt{n} \rceil.$$
Now we are going to prove that, when $m$ is a prime number and $n$ is a perfect square which is a multiple of $m^2(m-1)^2$, then this construction also satisfies that $\sum_{k=0}^{m-1}|\rho(V(G_{k}))|\ \varepsilon^k\neq 0$, and hence one may apply \Cref{pr:grafosafunciones}.

\begin{prop}\label{pr:construccionprimo}
    For every prime $p\geq 3$ and every perfect square $n=N^2$ with $N\equiv0\mod p(p-1)$, there exist $p$ induced subgraphs $G_0,G_1,\dots,G_{p-1}$ of $H(n,p)$ whose vertex sets partition $V(H(n,p))$ such that $\sum_{i=0}^{p-1}|\rho(V(G_{k}))|\ \varepsilon^k\neq 0$ and $(p-1)n-\min\{\delta(G_0),\delta(G_1),\dots,\delta(G_{p-1})\} = (p-1)\sqrt{n} .$ 
\end{prop}

For proving this result we will use some lemmas:

\begin{lem}\label{lm:valoresHmod} Let $m, \ell \geq 2$ and $s \in \mathbb Z$, and denote \[ H_s^\ell := \left|\left\lbrace (x_1,\dots,x_{\ell})\in \ [ 0, m-2]^{\ell}\ :\ \sum_{i=1}^{\ell}x_i\equiv s\mod m\right\rbrace\right|\ .\] Then,
 \[ H_s^\ell = \left\{ \begin{array}{lllll}
 \left(\sum_{i = 0}^{(\ell-2)/2} (m-1)^{2i}\right) (m-2) & \text{if} & \ell \text{ is even and } s + \ell \not\equiv 0 \mod m, \\
 \left(\sum_{i = 0}^{(\ell-2)/2} (m-1)^{2i}\right) (m-2) + 1 & \text{if} & \ell \text{ is even and } s + \ell \equiv 0 \mod m,  \\
 \left(\sum_{i = 0}^{(\ell-3)/2} (m-1)^{2i+1}\right) (m-2) + 1 & \text{if} & \ell \text{ is odd and } s + \ell \not\equiv 0 \mod m,  \\
 \left(\sum_{i = 0}^{(\ell-3)/2} (m-1)^{2i+1}\right) (m-2) & \text{if} & \ell \text{ is odd and } s + \ell \equiv 0 \mod m.   \end{array} \right. \]

\end{lem}

\begin{proof} The result follows by induction on $\ell$ applying the recursive formula  \[ H_s^\ell = \sum_{k = 0}^{m-2} H_{s-k}^{\ell-1} \] and the fact that for $\ell = 1$ one has that $H_s^1 = 1$ if $s \not\equiv -1 \mod m$, and  $H_s^1 = 0$ if $s \equiv -1 \mod m$, which lead us to the table in  \Cref{fig:tabla}, in which the entry in the $i$-th row and the $j$-th column is equal to $H_{j}^i$. It is clear that each row satisfies that its entries are all equal except one, and the one which is different corresponds to the case in which $i+j\equiv 0\mod m$.  
\end{proof}

 \begin{figure}[!h]
 \centering
 	\includegraphics[width=15.5cm]{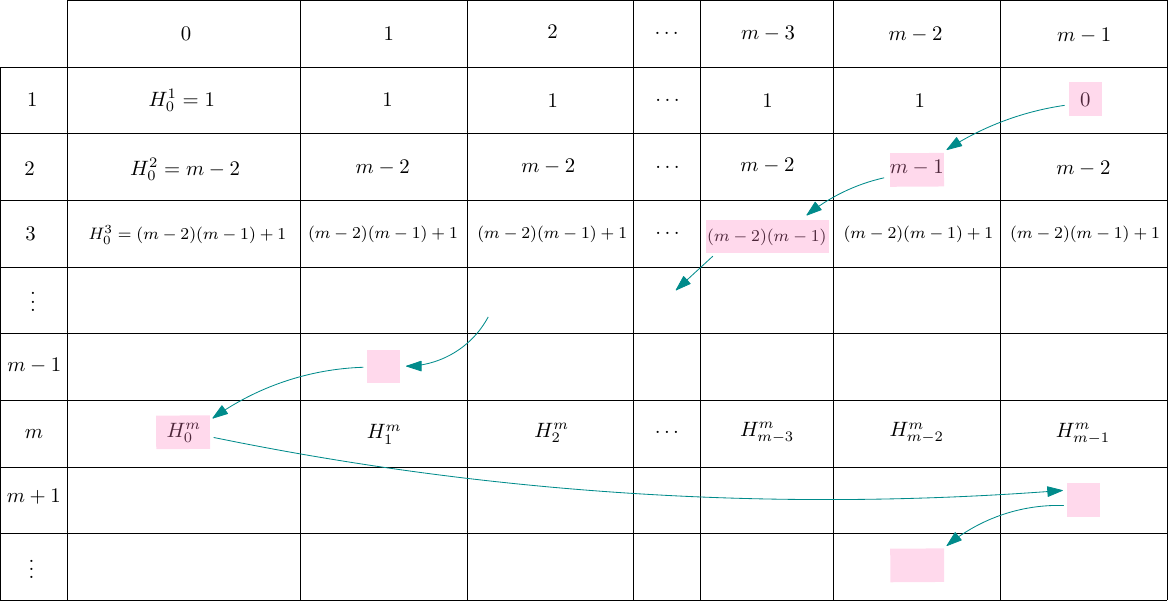}
  \caption{Table $T$ with $T_{ij}=H_{j}^i$}\label{fig:tabla}
 \end{figure}

\Cref{lm:valoresHmod} provides the exact values of $H_s^\ell$. However, in the proof of \Cref{pr:construccionprimo} we only use that $H_s^\ell - H_t^\ell \in \{-1,0,1\}$, and that  $H_s^\ell = H_t^\ell$ whenever $s + \ell \not\equiv 0 \mod m$ and $t + \ell \not\equiv 0 \mod m$.

A theorem due to Lucas (see, e.g.,~\cite{Fin47,Luc91}) provides the value of the binomial coefficient ${M \choose N}$ modulo a prime in terms of the $p$-adic expansion of $M$ and $N$. A direct consequence of this result is the following:

\begin{cor} \label{cor:lucas}
    Let $p$ be a prime and let $M \in \mathbb N$ be a multiple of $p$. If  $N \in \mathbb N$ is not a multiple of $p$, then ${M\choose N}\equiv 0 \mod p\ .$  As a consequence, if  $N_1,\ldots,N_s \in \mathbb N$ satisfy that $N_1 + \cdots + N_s = N$ and $N_i$ is not a multiple of $p$ for some $i \in [ 1, s]$, then   \[ {M \choose N_1,\ldots,N_s} \equiv 0 \mod p. \] 
\end{cor}

\begin{lem} \label{lm:deltamodp}Let $p \geq 3$ be a prime number, $t \in [p-2]$, $N$ a multiple of $p(p-1)$ and $s \in \mathbb Z$. Then, 
\[ \Delta_{t,s}^N  := \left| \left\lbrace (x_1,\ldots,x_N) \in [ 0, t ] ^N \setminus [ 1, t]^N \, : \, \sum_{i = 1}^N x_i \equiv s \mod p \right\rbrace \right| \] is a multiple of $p$.
\end{lem}
\begin{proof}
We observe that $\Delta_{t,s}^N = \Delta_{t,(s \mod p)}^N$, so we may assume that $0 \leq s \leq p-1$. We separate the proof in two cases.

We first assume that $s \neq 0$. We observe that for every $(x_1,\ldots,x_N) \in \Delta_{t,s}^N $ and every permutation $\sigma: \{1,\ldots,N\} \rightarrow \{1,\ldots,N\}$ one has that $(x_{\sigma(1)},\ldots,x_{\sigma(n)}) \in \Delta_{t,s}^N$. Hence, if $y_i$ denotes the number of entries among $x_1,\ldots,x_N$ that are equal to $i$, then there are $N \choose y_0,\ldots,y_t$ different elements that can be obtained permuting the entries of $(x_1,\ldots,x_N)$. Since $x_1 + \cdots + x_N = 0 y_0 + \cdots + t y_t \equiv s \ ({\rm mod}\  p)$, then there exists $j$ such that $y_j \not\equiv 0 \mod p$. Since $N$ is a multiple of $p$, applying \Cref{cor:lucas} we have that ${N \choose y_0,\ldots,y_t} \equiv 0 \mod p$.

We assume now that $s = 0$. We have that $\sum_{i = 0}^{p-1}  \Delta_{t,i}^N = (t+1)^N - t^N$ . Since $N$ is a multiple of $p-1$ and $1 \leq t < t + 1 \leq p-1,$ by Fermat's little Theorem one has that $(t+1)^N \equiv t^N \equiv 1 \mod p$. Then  $\Delta_{t,0}^N \equiv -  \sum_{i = 1}^{p-1} \Delta_{t,i}^N \equiv 0 \mod p$. 
\end{proof}

 \begin{proof}[Proof of  \Cref{pr:construccionprimo}.]
    Consider the partition of $[n]$ into the sets $A_1=\{1,\dots,N\}, A_2=\{N+1,\dots,2N\}, \ldots, A_N=\{N^2-N+1,\dots,N^2\}$ and consider the induced subgraphs $G_0,\ldots,G_{p-1}$ described in \Cref{lm:construcion} (see \Cref{fig:construction}).\begin{figure}[!h]
 \centering
 	\includegraphics[width=10cm]{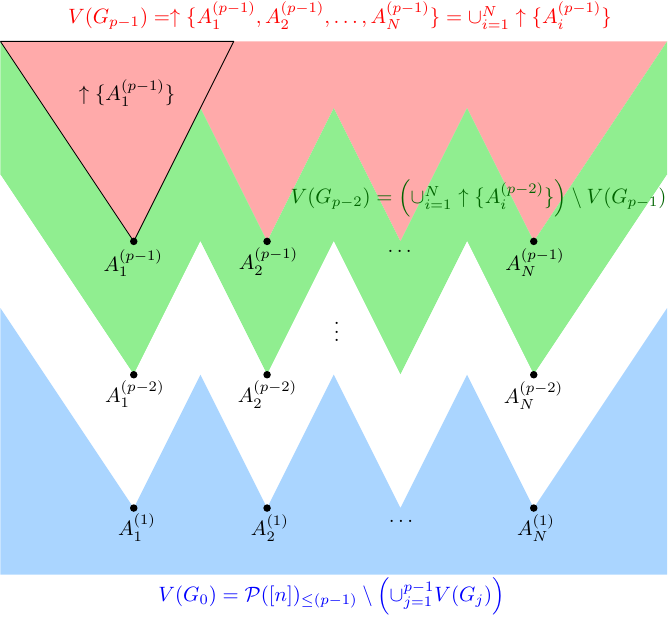}
 	\caption{$V(G_0),\dots,V(G_{p-1})$ as elements of the poset $\mathcal{P}([n])_{\leq(p-1)}$} 
 	\label{fig:construction}
 \end{figure}

By \Cref{lm:construcion} we have that $V(G_0),\dots,V(G_{p-1})$ partition $V(H(n,p))\cong\mathcal{P}([n])_{\leq(p-1)},$ and $(p-1)n-\min\{\delta(G_0),\delta(G_1),\dots,\delta(G_{p-1})\} = (p-1)\sqrt{n}$.

It only remains to prove that $\sum_{i=0}^{p-1}|\rho(V(G_{i}))|\ \varepsilon^i\neq 0$. Since $p$ is prime, by \Cref{rm:primo} this is equivalent to proving that $|\rho(V(G_0))|, \ldots, |\rho(V(G_{p-1}))|$ are not all equal. Indeed, we are going to show that $|\rho(V(G_{p-1}))| \not\equiv |\rho(V(G_{1}))| \mod p$. This will follow from:   
\begin{itemize}
    \item[(i)] $|\rho(V(G_{p-1})) \cap V(G_{p-1})| - |\rho(V(G_{1})) \cap V(G_{p-1})| \in \{1,-1\}$,
    \item[(ii)] $|\rho(V(G_{p-1})) \cap V(G_{0})| - |\rho(V(G_{1})) \cap V(G_{0})| = 0$, and
    \item[(iii)] $|\rho(V(G_{r})) \cap V(G_{i})| \equiv 0 \mod p$ for $1 \leq i \leq p-2$ and $0 \leq r \leq p-1$.  
\end{itemize}

We recall that, by \Cref{rm:additivenotation}, we have that  \[ \rho(V(G_i))=\cup_{j = 0}^{p-1} (D_j\cap V(G_{i-j}))\]
for every $i \in [0, p-1]$, where the indices are taken modulo $p$ and  $D_j=\{\X=(x_1,x_2,\dots,x_n)\in[0,p-1]^n\ |\ \sum_{i=1}^n x_i\equiv j\mod p\}$.

Let us prove (i), taking into account that \[ |\rho(V(G_{p-1})) \cap V(G_{p-1})| - |\rho(V(G_{1})) \cap V(G_{p-1})| =  |D_0 \cap V(G_{p-1})| - |D_2 \cap V(G_{p-1})|.\] We observe that if  $B \in \ \uparrow \{A_1^{(p-1)}\}$, then its last entry is free in the following sense: if we change its last entry by any value (or, in other words, if we change the multiplicity of $n$ in $B$), then we  obtain new elements of $\uparrow \{A_1^{(p-1)}\},$ and each of them belongs to a different $D_i$. In particular,  $|\uparrow \{A_1^{(p-1)}\} \cap D_0| - |\uparrow \{A_1^{(p-1)}\} \cap D_{2}| = 0$. The same happens if we take an element of $\uparrow \{A_i^{(p-1)}\}\setminus\left(\cup_{j=1}^{i-1}\uparrow \{A_j^{(p-1)}\}\right)$ for all $i\in\{2,\dots,N-1\}$, so we just have to focus on what happens with the elements of $\uparrow\{A_N^{(p-1)}\}\setminus\left(\cup_{i=1}^{N-1}\uparrow \{A_i^{(p-1)}\}\right)$.

The entries of every element of $V(H(n,p))$ can be separated in $N$ blocks according to the different sets $A_1,A_2,\dots,A_N$ to which their indices belong. If the first element of the block corresponding to $A_{N-1}$ is not equal to $p-1$, then the last entry of that block is free. So we can assume that this first element is equal to $p-1$, and repeating this reasoning we get that the elements of $V(G_{p-1})$ that are not equally distributed among $D_0,D_1,\dots,D_{p-1}$  must be of the form $$(\underbrace{p-1,\dots,p-1,*_1}_{N}\ ,\ \underbrace{p-1,\dots,p-1,*_2}_{N}\ ,\ \dots\ ,\ \underbrace{p-1,\dots,p-1,*_{N-1}}_{N}\ ,\ \underbrace{p-1,\dots,p-1,p-1}_{N}).\ $$ 
To compute which of these elements belong to $D_0$, we want that $\sum_{j=1}^{N-1}*_j+(p-1)(N-1)^2+ (p-1)N\equiv 0\mod p$. Since $N \equiv 0 \mod p$, this is equivalent to $\sum_{j=1}^{N-1}*_j\equiv 1 \mod p$, and following the notation of \Cref{lm:valoresHmod} there are exactly $H_{1}^{N-1}$ such vertices. Furthermore, for $D_2$, we want that $\sum_{j=1}^{N-1}*_j+(p-1)(N-1)^2+ (p-1)N\equiv 2\mod p$; i.e., we want that $\sum_{j=1}^{N-1}*_j\equiv 3 \mod p$, and there are exactly $H_{3}^{N-1}$ such vertices. Moreover, since $N$ is a multiple of $p$ one has that $N-1 + 1 \equiv 0 \mod p$ and, since $p \geq 3$, $N-1 + 3 \not\equiv 0 \mod p$; thus, by \Cref{lm:valoresHmod}, \[ |\rho(V(G_{p-1})) \cap V(G_{p-1})| - |\rho(V(G_{1})) \cap V(G_{p-1})| = H_1^{N-1} - H_3^{N-1} \in \{+1,-1\}.\]

Let us prove (ii), starting with the fact that \[ |\rho(V(G_{p-1})) \cap V(G_{0})| - |\rho(V(G_{1})) \cap V(G_{0})|  =  |D_{p-1} \cap V(G_{0})| - |D_1 \cap V(G_{0})|.\] As in the previous case, we separate the entries of these vertices in $N$ blocks. The vertices of $G_0$ have at least one entry equal to 0 in each of the $N$ blocks. If the first entry of the first block is equal to 0, then the last entry of that block is free. Hence we can assume that this entry is not equal to 0. Repeating this idea we get that the elements of $V(G_{0})$ that are not equally distributed among $D_0,D_1,\dots,D_{p-1}$ must be of the form $$(\underbrace{*_1^1,\dots,*_{N-1}^1,0}_{N}\ ,\ \underbrace{*_1^2,\dots,*_{N-1}^2,0}_{N}\ ,\ \dots\ ,\ \underbrace{*_1^N,\dots,*_{N-1}^N,0}_{N})\ ,$$ where $*_i^j \in[ 1, p-1]$. If we subtract one unit to every nonzero entry of the previous vector, then we get that \[ |D_{p-1} \cap V(G_{0})| - |D_1 \cap V(G_{0})| = H_{p-1}^{N(N-1)} - H_{1}^{N(N-1)},\] which equals $0$ by \Cref{lm:valoresHmod}.

Let us prove (iii). We are going to show that $|\rho(V(G_{r})) \cap V(G_{j})| \equiv 0 \mod p$ for all $r \in [ 0, p-1]$ and all $j\in[ 1, p-2]$ or, equivalently, that $|D_s \cap V(G_{j})| \equiv 0 \mod p$ for all $s \in [ 0, p-1]$ and all $j\in[ 1, p-2]$.
We observe that all the elements $\X \in V(G_j)$ can be uniquely built by choosing:
\begin{itemize}
\item $k \in [ 1, N ]$ (here $k$ indicates the smallest index such that $\X$ contains $A_k^{(j)}$), 
\item $c_1,\ldots,c_{k-1} \in [ 0,p-1]^N \setminus [ j,p-1 ]^N,$ 
\item $b \in [ j, p-1 ]^N \setminus [ j+1, p-1 ]^N$,
\item $c_{k+1},\ldots,c_N \in [ 0,p-1]^N \setminus [ j+1,p-1 ]^N,$ and
\end{itemize}
considering $\X = (c_1,\ldots,c_{k-1},b,c_{k+1}\ldots,c_N)$.

Hence, by \Cref{lm:deltamodp}, for a given $k \in [ 1, N ]$, $c_1,\ldots,c_{k-1} \in [ 0,p-1]^N \setminus [ j,p-1 ]^N,$ 
 $c_{k+1},\ldots,c_N \in [ 0,p-1]^N \setminus [ j+1,p-1 ]^N,$ the number of elements $\X = (c_1,\ldots,c_{k-1},b,c_{k+1}\ldots,c_N) \in D_s \cap V(G_j)$ is a multiple of $p$. Thus, $|D_s \cap V(G_{j})| \equiv 0 \mod p$.

By (i), (ii) and (iii) we can conclude that 
\[ |\rho(V(G_{p-1}))| - |\rho(V(G_{1}))| = \sum_{i = 0}^{p-1} (|\rho(V(G_{p-1})) \cap V(G_i)| - |\rho(V(G_{1})) \cap V(G_i)|) \equiv \pm 1 \mod p,\]
and hence $|\rho(V(G_{p-1}))| \neq |\rho(V(G_{1}))|$. 
 \end{proof}

\Cref{pr:construccionprimo} imposes  several conditions on the values of $m$ and $n$ ($m \geq 3$ is a prime and $n = N^2$ is a perfect square which is a multiple of $p^2(p-1)^2$), and our proof heavily relies on these hypotheses. However, we believe that the same type of results can be proved for general values of $m$ and $n$.

As a consequence of \Cref{pr:construccionprimo} together with  \Cref{pr:grafosafunciones} we have the following result, which in particular implies that if the $p$-ary sensitivity conjecture (\Cref{conj:sensit}) holds, then $c \geq 2$. Equivalently, this also implies that if \Cref{conj:sensitivity_graphs} holds for $p$-ary functions, then $\mu \leq 1/2.$
\begin{prop}\label{thm:quadratic}
    For every prime $p$ and every positive integer $D$, there exists $n_0$ such that for all $n \geq n_0$ there is a $p$-ary function $f:\{1,\varepsilon,\dots,\varepsilon^{p-1}\}^n\rightarrow\{1,\varepsilon,\dots,\varepsilon^{p-1}\}$ with $\deg(f) > D$ and  $$s(f)=\sqrt{(p-1)\deg(f)}\ .$$
\end{prop}
\begin{proof} Take $n_0 = N^2$ a perfect square such that $N \equiv 0 \mod p(p-1)$ and $n_0 > D/(p-1)$. By \Cref{pr:construccionprimo} together with  \Cref{pr:grafosafunciones}, there exists a $p$-ary function $g:\{1,\varepsilon,\dots,\varepsilon^{p-1}\}^{n_0}\rightarrow\{1,\varepsilon,\dots,\varepsilon^{p-1}\}$ of degree $\deg(g) = (p-1)n_0 > D$ and sensitivity $s(g)=\sqrt{(p-1)\deg(g)}$. Now, for $n \geq n_0$, we consider $f:\{1,\varepsilon,\dots,\varepsilon^{p-1}\}^n\rightarrow\{1,\varepsilon,\dots,\varepsilon^{p-1}\}$ given by $f(x_1,\dots,x_n)=g(x_1,\dots,x_{n_0})$ for every $(x_1,x_2,\dots,x_n)\in\{1,\varepsilon,\dots,\varepsilon^{p-1}\}^n$. Since $f$ and $g$ have the same sensitivity and degree, the result follows.
\end{proof}

This result together with Proposition \ref{pr:changeofT} yield the following.

\corquadratic*

\section{Conclusions}\label{last}

The main open problem is the $m$-ary sensitivity conjecture (\Cref{conj:sensit}) and its  reformulation in graph theoretical terms (\Cref{conj:sensitivity_graphs}).
The statement of  \Cref{conj:sensitivity_graphs} is a bit intricate due to the additional property concerning the values $|\rho(V(H_{\varepsilon^k}))|$. We present a stronger conjecture in more natural terms:

\begin{conj}[\textbf{Strong $m$-ary Sensitivity Conjecture}]\label{conj:stronger}Let $m$ be a positive integer and $\varepsilon$ an $m$-th primitive root of unity. There exists $\mu > 0$ such that for any partition of $H(n,m)$ into (possibly empty) induced subgraphs $H_1, \dots,H_{m}$ with $\sum_{i=1}^{m}|V(H_{i})|\ \varepsilon^i\neq 0$, we have $$\max\{\Delta(H_{1}),\dots,\Delta(H_{m})\} \in \Omega(n^{\mu})\ .$$
\end{conj}

Let us quickly see that indeed: 
\begin{prop}
    \Cref{conj:stronger} implies  \Cref{conj:sensitivity_graphs}. 
\end{prop}
\begin{proof}
        Assume that \Cref{conj:stronger} holds, and consider $m$ induced subgraphs $H_1,H_{\varepsilon},\dots,H_{\varepsilon^{m-1}}$ of $H(n,m)$ such that their vertex sets partition $V(H(n,m))$ and  $\sum_{i=0}^{m-1}|\rho(V(H_{\varepsilon^i}))|\varepsilon^i\neq 0$. Denote by $H_{\varepsilon^i}'$ the induced subgraph with $V(H_{\varepsilon^i}') = \rho(V(H_{\varepsilon^i})).$ Since we are assuming that \Cref{conj:stronger} holds,  then $$\max\{\Delta(H_1'),\dots,\Delta(H_{\varepsilon^{m-1}}')\} \in  \Omega(n^{\mu})$$ for some $\mu>0$.

    We observe that for all $v \in V(H(n,m))$ there exist $r,s \in [0,m-1]$ such that $v \in V(H_{\varepsilon^r})$, $v \in V(H_{\varepsilon^s}')$, and $\deg_{H_{\varepsilon^r}}(v) + \deg_{H_{\varepsilon^s}'}(v) \leq (m-1)n$ because the neighborhoods of $v$ in $H_{\varepsilon^r}$ and $H_{\varepsilon^s}'$ are disjoint. Hence,
    \[ (m-1)n-\min\{\delta(H_{\varepsilon^k}) \, \vert \, 0 \leq k \leq m-1\} \geq \max\{\Delta(H_1'),\dots,\Delta(H_{\varepsilon^{m-1}}')\} \]
    and the result follows.
\end{proof}

Note that if $m$ is a prime number, then the condition $\sum_{i=1}^{m}|V(H_{i})|\ \varepsilon^i\neq 0$ in~\Cref{conj:stronger} reduces to requiring that not all $H_i$ have the same order (see \Cref{rm:primo}). For $m=3$ this leads to the conjecture in the introduction:

\conjstrongertres*

Concerning \Cref{conj:stronger}, there are some recent works in which the authors study the maximum degrees of induced subgraphs of the Hamming graph $H(n,m)$. For example, denoting by $\alpha(G)$ the independence number of a graph $G$, Dong~\cite{Don21} constructed an induced subgraph of $H(n,m)$ on more than $\alpha(H(n,m)) = m^{n-1}$ vertices whose maximum degree is strictly smaller than $\sqrt{n}+1$. This construction was later improved by Tandya~\cite{Tan22}, who was able to prove that $H(n,m)$ has an induced subgraph on $m^{n-1}+1$ vertices with maximum degree $1$ whenever $m\geq 3$.

For $m = 3$, Potechin and Tsang~\cite{PT24} provide upper bounds on the maximum number of vertices of an induced subgraph with maximum degree at most $1$. They show that if $U \subseteq V(H(n,3))$ is disjoint from a maximum size independent set of $H(n,3)$ and the induced subgraph on $U$ has maximum degree at most $1$, then $|U|\leq 3^{n-1}+1$. This shows that a construction by García-Marco and Knauer~\cite{GMK22} is largest-possible. 
Potechin and Tsang also provide an induced subgraph of $H(n,3)$ with $3^{n-1}+18$ vertices with maximum degree equal to $1$ and which is almost optimal in several senses. 
However, it seems that from these constructions no partition into induced subgraphs of small maximum degrees can be obtained, if they are not allowed to be all of the same order. Finally, note that our graph theoretical Conjectures \ref{conj:strongertres}, \ref{conj:sensitivity_graphs},  \ref{conj:stronger} can all be seen as special variants of so-called defective colorings, see~\cite{Woo18} for a survey. However, our imbalance requirements on the color classes seem to be novel.

\end{document}